\documentclass{article}

\usepackage{answers}
\usepackage{xcolor}
\definecolor{sepia}{rgb}{0,0.3,0.2}

\usepackage{enumitem}
\usepackage[T1]{fontenc}
\usepackage[english]{babel}  
\usepackage{dsfont}
\usepackage{amsmath,amstext,am ssymb}
\usepackage{amsfonts}
\usepackage{amssymb}
\usepackage{graphicx}
\usepackage{amsthm}
\usepackage{stmaryrd}
\usepackage{listings}
\usepackage{makeidx}
\usepackage{xcolor}
\usepackage{authblk}
\usepackage{appendix}

\newtheorem{prop}{Proposition}
\newtheorem{corol}{Corollary}

\newtheorem{rmk}{Remark}

\newtheorem{lm}{Lemma}
\newtheorem{assumption}{Assumption}

\definecolor{colo}{rgb}{0,0,0.5}

\DeclareMathOperator{\V}{Var}

\DeclareMathOperator{\cov}{cov}

\newcommand{\R}{\mathbb{R}}

\newcommand{\N}{\mathbb{N}}

\newcommand{\E}{\mathbb{E}}
\newcommand{\PP}{\mathbb{P}}

\newcommand{\be}{\begin{eqnarray}}
\newcommand{\ee}{\end{eqnarray}}
\newcommand{\beq}{\begin{eqnarray*}}
\newcommand{\eeq}{\end{eqnarray*}}

\begin{document}
\author[1]{Baptiste Broto}

\author[2]{Fran\c{c}ois Bachoc }

\author[1]{Marine Depecker}

\author[3]{Jean-Marc Martinez}

\affil[1]{CEA, LIST, Universit\'e Paris-Saclay, F-91120, Palaiseau, France}
\affil[2]{Institut de Math\'ematiques de Toulouse, Universit\'e Paul Sabatier, F-31062 Toulouse, France}
\affil[3]{CEA, DES/DM2S, Universit\'e Paris-Saclay, F-91191 Gif-sur-Yvette, France}

\title{Gaussian linear approximation for the estimation of the Shapley effects}
\date\today

\maketitle

\begin{abstract}
  In this paper, we address the estimation of the sensitivity indices called "Shapley effects". These sensitivity indices enable to handle dependent input variables. The Shapley effects are generally difficult to estimate, but they are easily computable in the Gaussian linear framework. The aim of this work is to use the values of the Shapley effects in an approximated Gaussian linear framework as estimators of the true Shapley effects corresponding to a non-linear model. First, we assume that the input variables are Gaussian with small variances. We provide rates of convergence of the estimated Shapley effects to the true Shapley effects. Then, we focus on the case where the inputs are given by an non-Gaussian empirical mean. We prove that, under some mild assumptions, when the number of terms in the empirical mean increases, the difference between the true Shapley effects and the estimated Shapley effects given by the Gaussian linear approximation converges to 0. Our theoretical results are supported by numerical studies, showing that the Gaussian linear approximation is accurate and enables to decrease the computational time significantly.
\end{abstract}

\section{Introduction}

Sensitivity analysis, and particularly sensitivity indices, have became important tools in applied sciences. The aim of sensitivity indices is to quantify the impact of the input variables $X_1,\cdots,X_p$ on the output $Y=f(X_1,\cdots,X_p)$ of a model $f$. This information improves the interpretability of the model. In global sensitivity analysis, the input variables are assumed to be random variables. In this framework, the Sobol indices \cite{sobol_sensitivity_1993} were the first suggested indices to be applicable to general classes of models. Nevertheless, one of the most important limitations of these indices is the assumption of independence between the input variables. Hence, many variants of the Sobol indices have been suggested for dependent input variables \cite{mara_variance-based_2012,chastaing_indices_2013,mara_non-parametric_2015,chastaing2012generalized}.

Recently, Owen defined new sensitivity indices in \cite{owen_sobol_2014} called "Shapley effects". These sensitivity indices have many advantages over the Sobol indices for dependent inputs \cite{iooss2019shapley}. For general models, \cite{song_shapley_2016} suggested an estimator of the Shapley effects. However, this estimator requires to be able to generate samples with the conditional distributions of the input variables. Then, a consistent estimator has been suggested in \cite{broto_variance_2018}, requiring only a sample of the inputs-output. However, in practice, this estimator requires a large sample and is very costly in terms of computational time.\bigskip

Let us now consider the framework when the distribution of $X_1,\cdots,X_p$ is Gaussian and $f$ is linear, that we call the Gaussian linear framework. This framework is considered relatively commonly (see for example \cite{kawano_evaluation_2006,hammer_approximate_2011,rosti_linear_2004,clouvel_quantification_2019}), since the unknown function $f(X_1,\cdots,X_p)$ can be approximated by its linear
approximation around $\E(X)$. The Gaussian linear setting is highly beneficial, since the theoretical values of the Shapley effects can be computed explicitly \cite{owen_shapley_2017,iooss2019shapley,broto_sensitivity_2019,broto_block-diagonal_2019}. These values
depend on the covariance matrix  of the inputs and on
the coefficients of the linear model. An algorithm enabling to compute these values is implemented as the function  "ShapleyLinearGaussian" in the R package {\verb sensitivity } \cite{iooss_sensitivity_2020}. It is shown in \cite{broto_sensitivity_2019} that this computation is almost instantaneous when the number $p$ of input variables is smaller than 15, but becomes more difficult for $p\geq 25$. However, "ShapleyLinearGaussian" uses the possible block-diagonal structure of the covariance matrix to reduce the dimension, thereby reducing the computation cost \cite{broto_sensitivity_2019}.\\

The aim of this paper is to use the Shapley values computed from a Gaussian linear model as estimates of the true Shapley values corresponding to a non-linear model $f$. We provide convergence guarantees, as the Gaussian linear approximation is more and more accurate. We address the two following settings.

First, we assume that $X=(X_1,\ldots,X_p)$ is a Gaussian vector with variances decreasing to 0, and $f$ is not linear. We give the rate of convergence of the difference between the true Shapley effects and the ones given by the first-order Taylor polynomial of $f$ at the mean of $X$. To estimate the Shapley effects in a broader setting, we also provide the rate of convergence when the Taylor polynomial is unknown and the linear approximation is given by a finite difference approximation and a linear regression. To strengthen these theoretical results, we compare the three linear approximations on simulated data.

Second, we consider the case where the input vector is non-Gaussian and given by an empirical mean and the model $f$ is non-linear. We address the estimators of the Shapley values obtained by treating the input vector as Gaussian and the model as linear. We show that, as the number of summands goes to infinity, the estimators of the Shapley values converge to the true Shapley values, corresponding to the non-Gaussian input vector and the non-linear model. Then, we treat the particular case when the Shapley effects evaluate the impact of the individual estimation errors on a global estimation error. In numerical experiments, we compare the estimator of the Shapley effects given by the Gaussian linear framework with the estimator of the Shapley effects given by the general procedure of \cite{broto_variance_2018}, to the advantage of the former.

The rest of the article is organized as follows. In Section \ref{section_shapley}, we recall the definition of the Shapley effects and we detail the particular form of the Gaussian linear framework. 
Section \ref{section_approx_linear} provides the rates of convergence for Gaussian inputs and non-linear models. In Section \ref{section_approx_gaussian}, we address the case where the inputs are given by an empirical mean and $f$ is non-linear. The conclusions are given in Section \ref{section_conclu}. All the proofs are postponed to the supplementary material.

\section{The Shapley effects}\label{section_shapley}

Let $(X_i)_{i\in [1:p]}$ be random input variables on $\R^p$ and let $Y=f(X)$ be the real random output variable which is squared integrable . We assume that $\V(Y)\neq 0$. Here, $f:\R^p \to \R$ can be a numerical simulation model \cite{santner_design_2003}.

If $u\subset[1:p]$ and $x=(x_i)_{i \in [1:p]}\in \R^p$, we write $x_u:=(x_i)_{i\in u}$.
We can define the Shapley effects as in \cite{owen_sobol_2014}, where for each input  variable $X_i$, the Shapley effect is:
\begin{equation}\label{Shapley}
\eta_i(X,f):=\frac{1}{p\V(Y)}\sum_{u\subset -i}  \begin{pmatrix}
p-1\\ |u|
\end{pmatrix} ^{-1}\left(\V(\E(Y|X_{u\cup \{i\}}))-\V(\E(Y|X_u)) \right)
\end{equation}
where $-i$ is the set $[1:p]\setminus \{i\}$. We let $\eta(X,f)$ be the vector of dimension $p$ composed of $\eta_1(X,f),...,\eta_p(X,f)$. One can see in Equation \eqref{Shapley} that adding $X_i$ to $X_u$ changes the conditional expectation of $Y$, and increases the variability of this conditional expectation. The Shapley effect $\eta_i(X,f)$ is large when, on average, the variance of this conditional expectation increases significantly when $X_i$ is observed. Thus, a large Shapley effect $\eta_i(X,f)$ corresponds to an important input variable $X_i$.

The Shapley effects have interesting properties for global sensitivity analysis. Indeed, there is only one Shapley effect for each variable (contrary to the Sobol indices). Moreover, the sum of all the Shapley effects is equal to $1$ (see \cite{owen_sobol_2014}) and all these values lie in $[0,1]$ even with dependent inputs. This is very convenient for the interpretation of these sensitivity indices.\\

An estimator of the Shapley effects has been suggested in \cite{song_shapley_2016}. It is implemented in the R package {\verb sensitivity } as the function "shapleyPermRand". However, it requires to be able to generate samples with the conditional distributions of the inputs, which limits the application framework. \cite{broto_variance_2018} suggested another estimator which requires only a sample of the inputs-output. This estimator uses nearest-neighbour methods to mimic the generation of samples from these conditional distributions. It is implemented in the R package {\verb sensitivity } as the function "shapleySubsetMC". However, in practice, this estimator requires a large sample and is very costly in terms of computational time.

Consider now the case where $X \sim \mathcal{N}(\mu,\Sigma)$, with $\Sigma \in S_p^{++}(\R)$ and where the model is linear, that is $f:x\longmapsto \beta_0+\beta^T x$, for a fixed $\beta_0\in \R$ and a fixed vector $\beta$.
In this framework, the sensitivity indices can be calculated explicitly \cite{owen_shapley_2017}:
\begin{eqnarray}\label{eq_varianceShapley}
\eta_i(X,f):=\frac{1}{p\V(Y)}\sum_{u\subset -i}  \begin{pmatrix}
p-1\\ |u|
\end{pmatrix}^{-1}\left(\V(Y|X_u)-\V(Y|X_{u\cup \{i\}}) \right)
\end{eqnarray}
with
\begin{equation}\label{eq_V}
\V(Y|X_u)=\V(\beta_{-u}^T X_{-u}|X_u)=\beta_{-u}^T(\Sigma_{-u,-u}-\Sigma_{-u,u}\Sigma_{u,u}^{-1}\Sigma_{u,-u})\beta_{-u}
\end{equation}
where $\Gamma_{v,w}:=(\Gamma_{i,j})_{i\in v, j\in w}$. 
Thus, in the Gaussian linear framework, the Shapley effects are functions of the parameters $\beta$ and $\Sigma$. The Gaussian linear framework is thus very beneficial from an estimation point of view, because in general one needs to estimate conditional moments of the form $\V(\E(Y|X_v))$ for $v \subset [1,p]$, using nearest-neighbour methods, while in the Gaussian linear framework, only standard matrix vector operations are required.

\section{Approximation of a model by a linear model}\label{section_approx_linear}

\subsection{Introduction and notation}
To model uncertain physical values, it can be convenient to consider them as a Gaussian vector.  For example, the international libraries \cite{ENDF,JEFF,JENDL} on real data from the field of nuclear safety provide the average and covariance matrix of the input variables, so it is natural to model them with the Gaussian distribution. Hence, to quantify the impact of the uncertainties of the physical inputs of a model on a quantity of interest, it is commonly the case to estimate the Shapley effects of Gaussian inputs. The model $f$ is in general non-linear and the estimation procedures dedicated to non-linear models \cite{song_shapley_2016,broto_variance_2018} are typically computationally costly, with an accuracy that can be sensitive to the specific situation. Nevertheless, when the uncertainty on the inputs become small, the input vector converges to its mean $\mu$, and a linear approximation of the model at $\mu$ seems more and more appropriate.

To formalize this idea, let $X^{\{n\}}\sim\mathcal{N}(\mu^{\{n\} },\Sigma^{\{n\} })$ be the input vector, with a sequence of mean vectors $(\mu^{\{n\}})$ and a sequence of covariance matrices $(\Sigma^{\{n\}})$. The index $n$ can represent for instance the number of measures of an uncertain input, in which case the covariance matrix $\Sigma^{\{n\}}$ will decrease with $n$. 

\begin{assumption}\label{assum_approx_linear}
The covariance matrix $\Sigma^{\{n\} }$ decreases to $0$ such that the eigenvalues of $a^{\{n\}} \Sigma^{\{n\}}$ are lower-bounded and upper-bounded in $\R_+^*$, with $a^{\{n\}} \underset{n \to +\infty}{\longrightarrow}+\infty$. Moreover, $\mu^{\{n\}} \underset{n \to +\infty}{\longrightarrow}\mu$, where $\mu$ is a fixed vector.
\end{assumption}

In Assumption 1, the condition on the eigenvalues of $a^{\{n\}} \Sigma^{\{n\}}$ means that the correlation matrix obtained from $\Sigma^{\{n\}}$ can not get close to a singular matrix. This condition is necessary in our proofs.

If $j\in \N$ and if $f$ is $\mathcal{C}^{j}$ at $\mu^{\{n\}}$, we will write $f_j^{\{n\}}(x)=\frac{1}{j!} D^jf(\mu^{\{n\}})(x-\mu^{\{n\}})$ (where $D^j(\mu^{\{n\}})(z)$ is the image of $(z,z,\cdots,z)\in (\R^p)^j$ through the multilinear function $D^jf(\mu^{\{n\}})$, which gathers all the partial derivatives of order $j$ of $f$ at $\mu^{\{n\}}$) and $R_j^{\{n\}}(x)=f(x)-\sum_{l=0}^j f_{l}^{\{n\}}(x)$ the remainder of the $j$-th order Taylor approximation of $f$ at $\mu^{\{n\}}$. In particular, $f_1^{\{n\}}(x)=Df(\mu^{\{n\}})(x-\mu^{\{n\}})$, where $Df=D^1f$. We identify the linear function $Df(\mu^{\{n\}})$ with the corresponding row gradient vector of size $1\times p$ and the bilinear function $D^2f(\mu^{\{n\}})$ with the corresponding Hessian matrix of size $p\times p$. We also write $f_1(x)=Df(\mu)(x-\mu)$.

Finally, we assume that the function $f$ is subpolynomial, that is, there exist $k \in \N$ and $C>0$ such that,
$$
\forall x \in \R^p,\; |f(x)| \leq C(1+\|x\|^k).
$$

\subsection{Theoretical results}

\subsubsection{First-order Taylor polynomial}

First, we study the asymptotic difference between the Shapley effects given by the true model $f$ and the ones given by the first-order Taylor polynomial of $f$ at $\mu^{\{n\}}$. Remark that adding a constant to the function does not affect the values of the Shapley effects. Thus, the Shapley effects $\eta(X^{\{n\}},f(\mu^{\{n\}})+f_1^{\{n\}})$ given by the first-order Taylor polynomial of $f$ at $\mu^{\{n\}}$ are equal to $\eta(X^{\{n\}},f_1^{\{n\}})$. In the next proposition, we show that approximating the true Shapley effects of the non-linear $f$ by the Shapley effects of the linear approximation $f_1^{\{n\}}$ yields a vanishing error of order $1 / a^{\{n\}}$ as $n \to \infty$.

\begin{prop}\label{prop_linear1}
Assume that $X^{\{n\}}\sim \mathcal{N}(\mu^{\{n\}},\Sigma^{\{n\}})$, Assumption \ref{assum_approx_linear} holds and $f$ is subpolynomial and $\mathcal{C}^3$ on a neighbourhood of $\mu$ and $Df(\mu)\neq 0$. Then,
$$
\|\eta(X^{\{n\}}, f)-\eta(X^{\{n\}},f_1^{\{n\}})\| = O\left(\frac{1}{a^{\{n\}}}\right).
$$
\end{prop}

We remark that, when $f$ is a computer model, it can be the case that the gradient vector is available. First, the computer model can already provide them, by means of the Adjoint Sensitivity Method \cite{cacuci2003sensitivity}. Second, automatic differentiation methods can be used on the source file of the code and yield a differentiated code \cite{hascoet2004tapenade}.

\begin{rmk}\label{rmk}
The rate $O(1\slash a^{\{n\}})$ is the best rate that we can reach under the assumptions of Proposition \ref{prop_linear1}. Indeed, letting $X^{\{n\}}=(X^{\{n\}}_1,X^{\{n\}}_2)\sim \mathcal{N}(0,\frac{1}{a^{\{n\}}} I_2)$ and $Y^{\{n\}}=f(X^{\{n\}})=X_1^{\{n\}}+X_2^{\{n\}2}$, we have  $\eta_1(X^{\{n\}},f_1^{\{n\}})=1$ and $\eta_2(X^{\{n\}},f_1^{\{n\}}) =0$. Moreover, $\eta_1(X^{\{n\}},f)= \frac{a^{\{n\}}}{a^{\{n\}}+2}$ and $
\eta_2(X^{\{n\}},f)=\frac{2}{a^{\{n\}}+2}$.
Thus, the rate of the difference between $\eta(X^{\{n\}}, f)$ and $\eta(X^{\{n\}},f_1^{\{n\}})$ is exactly $1\slash a^{\{n\}}$.
\end{rmk}

In Proposition \ref{prop_linear1}, we bound the difference between the Shapley effects given by $f$ and the ones given by the first-order Taylor polynomial of $f$. Moreover, when the matrix $a^{\{n\}}\Sigma^{\{n\}}$ converges, Proposition \ref{prop_linear2} shows that the Shapley effects given by the Taylor polynomial converge.

\begin{prop}\label{prop_linear2}
Assume that $X^{\{n\}}\sim \mathcal{N}(\mu^{\{n\}},\Sigma^{\{n\}})$, Assumption \ref{assum_approx_linear} holds, $f$ is $\mathcal{C}^1$ on a neighbourhood of $\mu$, $Df(\mu)\neq 0$ and $a^{\{n\}} \Sigma^{\{n\}} \underset{n \to +\infty}{\longrightarrow}\Sigma \in S_p^{++}(\R)$. Then, if $X^*\sim \mathcal{N}(\mu,\Sigma)$,
$$ 
\|\eta(X^{\{n\}},f_1^{\{n\}})-\eta(X^*,f_1)\|=O(\| \mu^{\{n\}}-\mu\|)+O(\|a^{\{n\} }\Sigma^{\{n\}}-\Sigma\| ).
$$
\end{prop}

Proposition \ref{prop_linear1} shows that replacing $f$ by its first-order Taylor polynomial $f_1^{\{n\}}$ does not impact significantly the Shapley effects when the input variances are small. Thus, the knowledge of $f_1^{\{n\}}$ would enable us to use the explicit expression \eqref{eq_V} of the Gaussian linear case, and for instance the  function "ShapleyLinearGaussian" of the package {\verb sensitivity }, to estimate the true Shapley effects $\eta(X^{\{n\}},f)$. However, in practice, the first-order Taylor polynomial $f_1^{\{n\}}$ is not always available, except for instance in situations described above. Thus, one may be interested in replacing the true first-order Taylor polynomial $f_1^{\{n\}}$ by an approximation. We will study two such approximations given by finite difference and linear regression.

\subsubsection{Finite difference approximation}\label{section_linear_approx_finite}

For $h=(h_1,\cdots, h_p)\in (\R_+^*)^p$ and writing $(e_1,\cdots,e_p)$ the canonical basis of $\R^p$, let 
\begin{equation}\label{eq_diff_finie}
\widehat{D}_h f(x):=\left( \frac{f\left( x+e_1 h_1\right)-f\left( x-e_1 h_1\right)}{2h_1}, \cdots,\frac{f\left( x+e_p h_p\right)-f\left( x-e_p h_p\right)}{2h_p}\right),
\end{equation}
be the approximation of the differential of $f$ at $x$ with the steps $h_1,\cdots,h_p$. 
If $(h^{\{n\}})_n$ is a sequence of $(\R_+^*)^p$ converging to $0$, let
$$
\tilde{f}_{1,h^{\{n\}}}(x):=\tilde{f}_{1,h^{\{n\}},\mu^{\{n\}}}(x):=\widehat{D}_{h^{\{n\}}} f(\mu^{\{n\}})(x-\mu^{\{n\}})
$$
be the approximation of the first-order Taylor polynomial of $f-f(\mu^{\{n\}})$ at $\mu^{\{n\}}$ with the steps $h_1,\cdots,h_p$. The next proposition ensures that the Shapley effects computed from the true Taylor polynomial and the approximated one are close, for small steps.

\begin{prop}\label{prop_linear_Dh}
Under the assumptions of Proposition \ref{prop_linear1}, we have
$$
\|\eta(X^{\{n\}},f_1^{\{n\}})-\eta(X^{\{n\}},\tilde{f}_{1,h^{\{n\}}}^{\{n\}}) \|=O\left( \|h^{\{n\}}\|^2\right).
$$
\end{prop}

Then, the next corollary extends Propositions \ref{prop_linear1} and \ref{prop_linear2} to the approximated Taylor polynomial based on finite differences.

\begin{corol}\label{corol_linear_Dh1}
Under the assumptions of Proposition \ref{prop_linear1}, and if $\| h^{\{n\}}\|\leq  \frac{C_{\sup}}{\sqrt{a^{\{n\}}}}$ (for example, choosing $h_i^{\{n\}}:=\sqrt{\V(X_i^{\{n\}})}$, the standard deviation of $X_i^{\{n\}}$), we have
$$
\|\eta(X^{\{n\}},f)-\eta(X^{\{n\}},\tilde{f}_{1,h^{\{n\}}}^{\{n\}})\|=O(\frac{1}{a^{\{n\}}}).
$$
Moreover, if $a^{\{n\}} \Sigma^{\{n\} } \underset{n \to +\infty}{\longrightarrow}\Sigma$, then, letting $X^*\sim \mathcal{N}(\mu,\Sigma)$,
$$
\| \eta(X^{\{n\}},\tilde{f}_{1,h^{\{n\}}}^{\{n\}}) - \eta(X^*,f_1)\|=O(\| \mu^{\{n\}}-\mu\|)+O(\| a^{\{n\}}\Sigma^{\{n\}}-\Sigma\| )+ O\left(\frac{1}{a^{\{n\}}}\right).
$$
\end{corol}

\subsubsection{Linear regression}

For $n\in \N$ and $N\in \N^*$, let $(X^{\{n\}(l)})_{l\in [1:N]}$ be an i.i.d. sample of $X^{\{n\}}$ of size $N$ and assume that we compute the image of $f$ at each sample point, obtaining the vector $Y^{\{n\}}$. Then, we can approximate $f$ with a linear regression, by least squares. In this case, we estimate the coefficients of the linear regression by the vector:
$$
\begin{pmatrix}
\widehat{\beta}_0^{\{n\} }\\
\widehat{\beta}^{\{n\}}
\end{pmatrix}
=\left(A^{\{n\}T}A^{\{n\}}\right)^{-1} A^{\{n\}T} Y^{\{n\}},
$$
where $A^{\{n\}}\in \mathcal{M}_{N,p+1}(\R)$ is such that, for all $j\in [1:N]$, the $j$-th line of $A^{\{n\}}$ is $(1\; X^{\{n\}(j)T})$.
The function $f$ is then approximated by 
$$
\widehat{f}_{lin}^{\{n\}(N)}: x \longmapsto \widehat{\beta}^{\{n\}}_0 + \widehat{\beta}^{\{n\}T} x.
$$
Remark that the linear function $\widehat{f}_{lin}^{\{n\}(N)}$ is random and so, the deduced Shapley effects $\eta(X^{\{n\}},\widehat{f}_{lin}^{\{n\}(N)})$ are random variables.
The next proposition and corollary correspond to Proposition \ref{prop_linear_Dh} and Corollary \ref{corol_linear_Dh1}, for the linear regression approximation of $f$.

\begin{prop}\label{prop_linear_regression}
Under Assumption \ref{assum_approx_linear}, if $f$ is $\mathcal{C}^2$ on a neighbourhood of $\mu$ with $Df(\mu)\neq 0$, there exist $C_{\inf}>0$, $C_{\sup}^{(1)}<+\infty$ and $C_{\sup}^{(2)}<+\infty$ such that, with probability at least $1 -C_{\sup}^{(1)}\exp(-C_{\inf} N)$, we have
$$
\|\eta(X^{\{n\}},f_1^{\{n\}})-\eta(X^{\{n\}},\widehat{f}_{lin}^{\{n\}(N)})\|\leq C_{\sup}^{(2)} \frac{1}{\sqrt{a^{\{n\}}}} .
$$
\end{prop}

\begin{corol}
Under the assumptions of Proposition \ref{prop_linear1}, there exist $C_{\inf}>0$, $C_{\sup}^{(1)}<+\infty$ and $C_{\sup}^{(2)}<+\infty$ such that, with probability at least $1 -C_{\sup}^{(1)}\exp(-C_{\inf} N)$, we have
$$
\|\eta(X^{\{n\}},f)-\eta(X^{\{n\}},\widehat{f}_{lin}^{\{n\}(N)})\|\leq C_{\sup}^{(2)} \frac{1}{\sqrt{a^{\{n\}}}} .
$$
Moreover, if $a^{\{n\}} \Sigma^{\{n\} } \underset{n \to +\infty}{\longrightarrow}\Sigma$, then, letting $X^*\sim \mathcal{N}(\mu,\Sigma)$, there exists $C_{\sup}^{(3)}<+\infty$ such that,  with probability at least $1 -C_{\sup}^{(1)}\exp(-C_{\inf} N)$,
$$
\| \eta(X^{\{n\}},\widehat{f}_{lin}^{\{n\}(N)}) - \eta(X^*,f_1)\|\leq C_{\sup}^{(3)}\left( \| \mu^{\{n\}}-\mu\|+ \| a^{\{n\}}\Sigma^{\{n\}}-\Sigma\| + \frac{1}{\sqrt{a^{\{n\}}}}\right).
$$
\end{corol}

\subsection{Numerical experiments}

In this section, we compute the Shapley effects of the true function $f$ and the ones obtained from the three previous linear approximations to illustrate the previous theoretical results.
Let $p=4$ and 
$$
f(x)=\cos(x_1)x_2+\sin(x_2)+2\cos(x_3)x_1-\sin(x_4).
$$
This function is $1$-Lipschitz continuous and $\mathcal{C}^{\infty}$ on $\R^4$. We choose $\Sigma^{\{n\}}=\frac{1}{n^2}\Sigma$ (that is, $a^{\{n\} }=n^2$), where $\Sigma$ is defined by:
$$
\Sigma= A^T A,\;\;\; A=\begin{pmatrix}
-2 & -1 & 0 & 1 \\ 2 & -2 & -1 & 0 \\ 1 & 2 & -2 & -1\\ 0 & 1 & 2 & -2
\end{pmatrix}.
$$
Let $\mu=(1, 0 ,2 ,1)$ and $\mu^{\{n\}}=\mu+\frac{1}{n}(1,1,1,1)$.

\begin{figure}[ht]
    \centering
    \includegraphics[width=13cm,height=10cm]{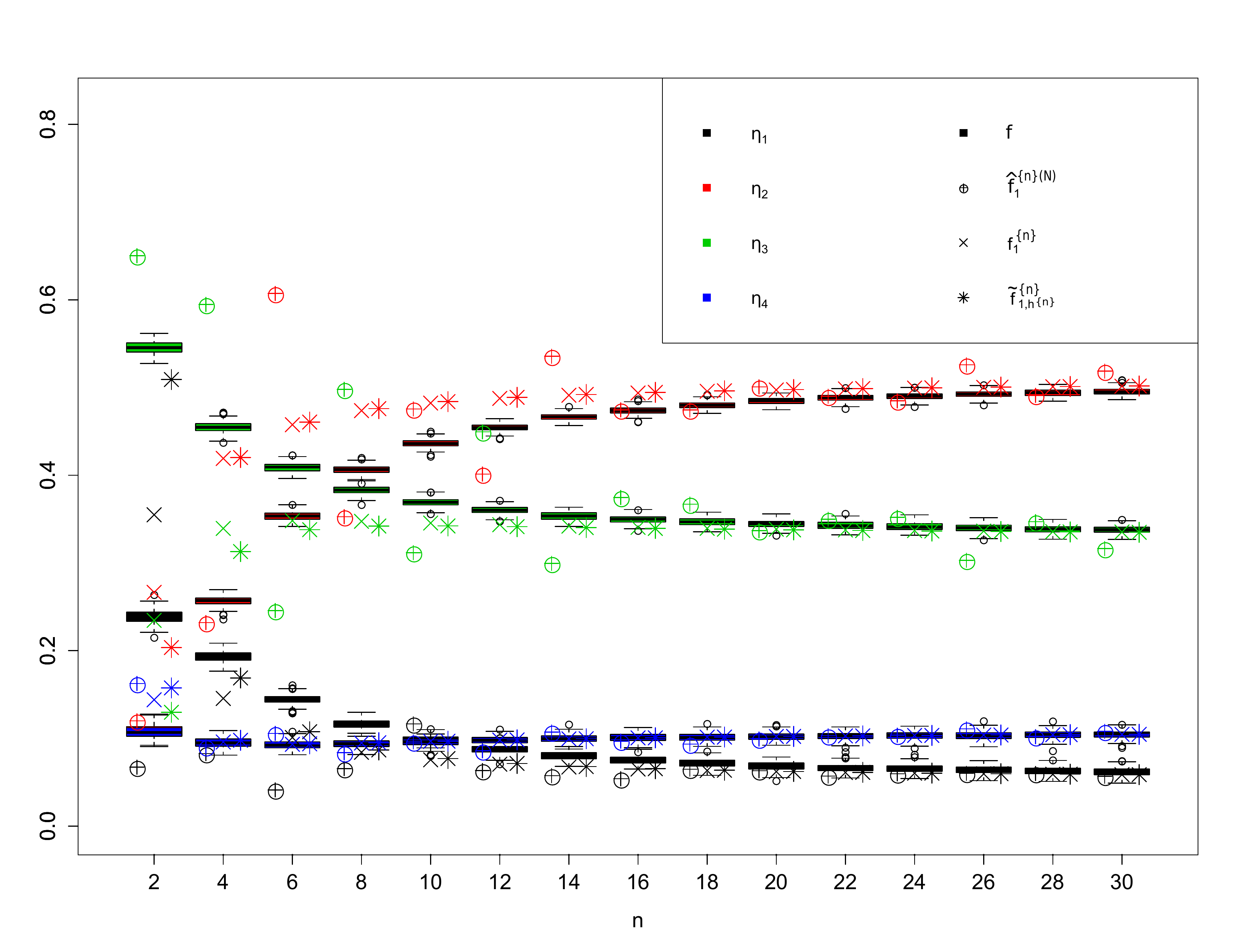}
    \caption{Shapley effects of the linear approximations $\widehat{f}_{lin}^{\{n\}(N)}$, $f_1^{\{n\}}$, $\tilde{f}_{1,h^{\{n\}}}^{\{n\}}$ and boxplots of estimates of the Shapley effects of the function $f$.}
    \label{fig_linear}
\end{figure}

On Figure \ref{fig_linear}, we plot, for different values of $n$, the vector $\eta(X^{\{n\}},\widehat{f}_{lin}^{\{n\}(N)})$ (given by the linear regression), the vector $\eta(X^{\{n\}},f_1^{\{n\}})$ (given by the true Taylor polynomial), the vector $\eta(X^{\{n\}},\tilde{f}_{1,h^{\{n\}}}^{\{n\}})$ (given by the finite difference approximation of the derivatives) and the boxplots of 200 estimates of $\eta(X^{\{n\}},f)$ computed by the R function "shapleyPermRand" from the R package {\verb sensitivity } (see \cite{song_shapley_2016,iooss2019shapley}), which is adapted to non-linear functions, with parameters $N_V=10^5$, $m=10^3$ and $N_I=3$. To compute the linear regression, we observed a sample of size $N=40$. To compute the finite difference approximation, we took $h_i^{\{n\}}= \sqrt{\V(X_i^{\{n\}})}$.

The differences between the Shapley effects given by $f$ and the ones given by the linear approximations of $f$ seem to converge to $0$, as it is proved by Propositions \ref{prop_linear1}, \ref{prop_linear_Dh} and \ref{prop_linear_regression}. Moreover, Figure \ref{fig_linear} emphasizes that the Shapley effects obtained from the linear regression get closer slower to the true ones than the ones given by the other linear approximations.

We remark that we have here $\Sigma^{\{n\}}=\frac{1}{a^{\{n\}}}\Sigma$ and thus the assumptions of Proposition \ref{prop_linear2} hold. Hence, the values of the true Shapley effects $\eta(X^{\{n\}},f)$ converge, as we can see on Figure \ref{fig_linear}.

The computation time for each estimate of the Shapley effects is around 5 seconds using "shapleyPermRand", $1.9\times 10^{-3}$ using the linear approximation $f_1^{\{n\}}$ or $\tilde{f}_{1,h^{\{n\}}}^{\{n\}}$ and $2.4\times 10^{-3}$ using the linear approximation $\widehat{f}_{lin}^{\{n\}(N)}$.
Remark that this time difference can become more accentuated if the function $f$ is a costly computer code.

\section{Approximation of the empirical mean by a Gaussian vector}\label{section_approx_gaussian}

\subsection{Theoretical results}\label{section_approx_gaussian_theore}

Here, we extend the results of Section \ref{section_approx_linear} to the case where the distribution of the input (that we now write  $\widehat{X}^{\{n\}}$) is close to a Gaussian distribution $X^{\{n\}}$. We focus on the setting where the input vector is an empirical mean
$$
\widehat{X}^{\{n\}}=\frac{1}{n}\sum_{l=1}^n U^{(l)},
$$
where $(U^{(l)})_{l\in [1:n]}$ is an i.i.d. sample of a random vector $U$ in $\R^p$ such that $\E(\|U\|^2)<+\infty$ and $\V(U)\neq 0$. Let $\mu:=\E(U)$ and $\Sigma$ be the covariance matrix of $U$. Remark that, as is Section \ref{section_approx_linear}, the input vector $\widehat{X}^{\{n\}}$ is a random vector converging to its mean, and its covariance matrix $\Sigma^{\{n\}}$ is equal to $\frac{1}{n}\Sigma$.

Contrary to Section \ref{section_approx_linear}, $\widehat{X}^{\{n\}}$ is not Gaussian, but, thanks to the central limit theorem, its distribution is close to $\mathcal{N}(\mu, \frac{1}{n}\Sigma)$. Hence,
we would like to estimate the Shapley effects $\eta(\widehat{X}^{\{n\}},f)$ by $\eta(X^*,Df(\mu))$, where $X^*\sim \mathcal{N}(0,\Sigma)$, since $\eta(X^*,Df(\mu))$ can be computed using the explicit expression \eqref{eq_V} of the Gaussian linear case, and for instance the function "ShapleyLinearGaussian" of the package {\verb sensitivity }.

\begin{prop}\label{prop_approx_gaussien}
Assume that $f$ is $\mathcal{C}^3$ on a neighbourhood of $\mu$ with $Df(\mu)\neq 0$ and that $f$ is subpolynomial, that is there exist $k\in \N^*$ and $C>0$ such that for all $x\in \R^p$, we have $|f(x)|\leq C(1+\|x\|^k)$. If $\E(\|U\|^{4k})<+\infty$ and if $U$ has a bounded probability density function, then
$$
\eta(\widehat{X}^{\{n\}},f)\underset{n \to +\infty}{\longrightarrow}\eta(X^*,Df(\mu)).
$$
\end{prop}

Proposition \ref{prop_approx_gaussien} justifies that $\eta(X^*,Df(\mu))$ is a good approximation of $\eta(\widehat{X}^{\{n\}},f)$. Furthermore, if $\mu$, $\Sigma$ and $Df(\mu)$ are unknown, the following corollary shows that they can be replaced by approximations. Let $(U^{\{l\}\prime})_{l\in [1:n']}$ and $(U^{\{l\}\prime \prime})_{l\in [1:n'']}$  be independent of $(U^{\{l\}})_{l\in [1:n]}$, composed of i.i.d. copies of $U$ and with $n'=n'(n)$ and $n'' = n''(n)$ such that $n', n'' \to \infty$ when $n \to \infty$. We can estimate $\mu$ (resp. $\Sigma$) by the empirical mean $\widehat{X}^{\{n'\}\prime}$ of $(U^{\{l\}\prime})_{l\in [1:n']}$ (resp. the empirical covariance matrix $\widehat{\Sigma}^{\{n''\}\prime \prime}$ of $(U^{\{l\}\prime \prime})_{l\in [1:n'']}$), and we can estimate $Df$ by a finite difference approximation. The next corollary guarantees that the error stemming from these additional estimations goes to $0$ as $n \to \infty$.

\begin{corol}\label{corol_approx_gaussien}
Assume that the assumptions of Proposition \ref{prop_approx_gaussien} hold and that $(h^{\{n\}})_{n\in \N}$ is a sequence of $(\R_+^*)^p$ converging to $0$. Let $X^{*n}$ be a random vector with distribution $\mathcal{N}(\mu,\widehat{\Sigma}^{\{n''\}\prime \prime})$ conditionally to $\widehat{\Sigma}^{\{n''\}\prime \prime}$. Then
    $$
    \left\|\eta(\widehat{X}^{\{n\}},f)- \eta(X^{*n},\tilde{f}_{1,h^{\{n\}},\widehat{X}^{\{n'\}\prime}}^{\{n\}})\right\| \overset{a.s.}{\underset{n \to +\infty}{\longrightarrow}}0,
    $$
where $\tilde{f}_{1,h^{\{n\}},\widehat{X}^{\{n'\}\prime}}^{\{n\}}$ is the linear approximation of $f$ at $\widehat{X}^{\{n'\}\prime}$ obtained from Equation \eqref{eq_diff_finie} by replacing $\mu^{\{n\}}$ by $\widehat{X}^{\{n'\}'}$.
\end{corol}

\begin{rmk}\label{rmk_corol_approx_gaussien}
If $\mu$, $\Sigma$ or $Df$ is known, the previous corollary holds replacing $\widehat{X}^{\{n'\}\prime},\widehat{\Sigma}^{\{n''\}\prime \prime}$ or $\tilde{f}_{1,h^{\{n\}},\widehat{X}^{\{n'\}\prime}}^{\{n\}}$ by $\mu,\Sigma$ or $Df(\widehat{X}^{\{n'\}\prime})$ respectively. 
\end{rmk}

\begin{rmk}
The notation $\eta(X^{*n},\tilde{f}_{1,h^{\{n\}},\widehat{X}^{\{n'\}\prime}}^{\{n\}})$ is to be understood conditionally to $\widehat{\Sigma}^{\{n''\}\prime \prime}, \widehat{X}^{\{n'\}\prime}$. That is, conditionally to  $\widehat{\Sigma}^{\{n''\}\prime \prime}, \widehat{X}^{\{n'\}\prime}$, the Shapley effects $\eta(X^{*n},\tilde{f}_{1,h^{\{n\}},\widehat{X}^{\{n'\}\prime}}^{\{n\}})$ are defined with the fixed linear function $\tilde{f}_{1,h^{\{n\}},\widehat{X}^{\{n'\}\prime}}^{\{n\}}$ and the Gaussian distribution for $X^{*n}$.
\end{rmk}

\subsection{Application to the impact of individual estimation errors}\label{section_approx_gaussian_appli}

Let us show an example of application of the results of Section \ref{section_approx_gaussian_theore}. Let $U$ be a continuous random vector of $\R^p$, with a bounded density and with an unknown mean $\mu$. Assume that we observe an i.i.d. sample $(U^{(l)})_{l\in [1:n]}$ of $U$ and that we focus on the estimation of a parameter $\theta=f(\mu)$, where $f$ is $\mathcal{C}^3$. This parameter is estimated by  $f(\widehat{X}^{\{n\}})$ (which is asymptotically efficient by the delta-method), where $\widehat{X}^{\{n\}}$ is the empirical mean of $(U^{(l)})_{l\in [1:n]}$. The estimation error of each variable $\widehat{X}_i^{\{n\}}$ (for $i=1,\cdots,p$) propagates through $f$. To quantify the part of the estimation error of $Y=f(\widehat{X}^{\{n\}})$ caused by the individual estimation errors of each $\widehat{X}_i^{\{n\}}$ (for $i=1,\cdots,p$), one can estimate the Shapley effects $\eta(\widehat{X}^{\{n\}},f)=\eta(\widehat{X}^{\{n\}}-\mu,f(\cdot +\mu)-f(\mu))$ which assess the impact of individual errors on the global error.
To that end, Proposition \ref{prop_approx_gaussien} and Corollary \ref{corol_approx_gaussien} state that the Shapley effects can be estimated using a Gaussian linear approximation, with an error that vanishes as $n$ increases.\\

For example, let $f=\| \cdot \|^2$ and $p=5$. In this case, the derivative $Df$ is known and no finite difference approximation is required. To generate $U$ with a bounded density and with dependencies, we define $A_1\sim \mathcal{U}([5,10])$, $A_2\sim \mathcal{N}(0,4)$, $A_3$ with a symmetric triangular distribution $T(-1,8)$, $A_4 \sim 5 Beta(1,2)$ and $A_5\sim Exp(1)$. Then, we define
$$
\left\{ \begin{array}{cc}
U_1 &=A_1+2A_2 -0.5 A_3   \\
U_2 &= A_2+2A_1-0.5A_5\\
U_3 &= A_3+2 A_2 -0.5 A_5\\
U_4 &= A_4+2A_1-0.5A_2\\
U_5 &=A_5 +2 A_3 -0.5A_4 .
\end{array} \right.
$$
Since the mean $\mu$ and the covariance matrix $\Sigma$ are unknown, we need to estimate them (as in Corollary \ref{corol_approx_gaussien}). Using the notation of Section \ref{section_approx_gaussian_theore}, we choose $n=n'=n''$ and $(U^{(l)\prime})_{l\in [1:n']}=(U^{(l)\prime \prime})_{l\in [1:n']}$ (that is, we estimate the empirical mean and the empirical covariance matrix with the same sample). We estimate the Shapley effects $\eta(\widehat{X}^{\{n\}},f)$ by $\eta(X^{*n},Df(\widehat{X}^{\{n\}\prime}))$, where $X^{*n}$ is a random vector with distribution $\mathcal{N}(\mu,\widehat{\Sigma}^{\{n\}\prime \prime})$ conditionally to $\widehat{\Sigma}^{\{n\} \prime \prime}$. By Corollary \ref{corol_approx_gaussien} and Remark \ref{rmk_corol_approx_gaussien}, the difference between $\eta(\widehat{X}^{\{n\}},f)$ and $\eta(X^{*n},Df(\widehat{X}^{\{n\}\prime}))$ converges to 0 almost surely when $n$ goes to $+\infty$.

Here, we compute 1000 estimates of $\mu$ and $\Sigma$ and we compute the 1000 corresponding Shapley effects of the Gaussian linear approximation $\eta(X^{*n},Df(\widehat{X}^{\{n\}\prime}))$. To compare with these estimates, we also compute 1000 estimates given by the function "shapleySubsetMC" suggested in \cite{broto_variance_2018}, with parameters $N_{tot}=1000$, $N_i=3$ and with an i.i.d. sample of $\widehat{X}^{\{n\}}$ with size 1000. We plot the results on Figure \ref{fig:linear_general}.\\

\begin{figure}[ht]
    \centering
    \includegraphics[scale=0.4]{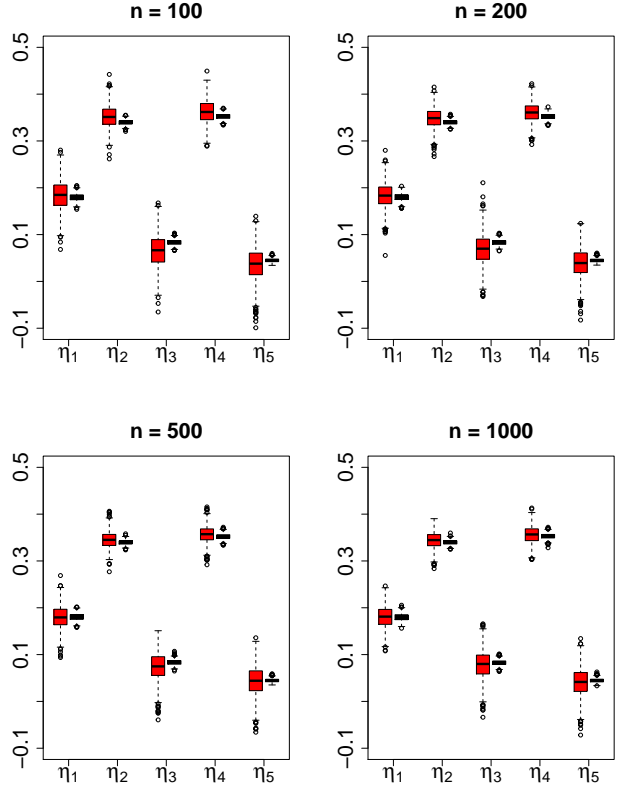}
    \caption{Boxplots of the estimates of the Shapley effects given by the general estimation function "shapleySubsetMC" (in red) and by the Gaussian linear approximation (in black).}
    \label{fig:linear_general}
\end{figure}

We observe that the estimates of the Shapley effects given by "shapleySubsetMC" and the Gaussian linear approximation are rather similar, even for $n=100$. However, the variance of the estimates given by the Gaussian linear approximation is smaller than the one of the general estimates given by "shapleySubsetMC". Moreover, each Gaussian linear estimation requires only a sample of $(U^{(l)\prime})_{l\in [1:n]}$ (to compute $\widehat{X}^{\{n\}\prime}$ and $\widehat{\Sigma}^{\{n\}\prime \prime}$) and takes around 0.007 second on a personal computer, whereas each general estimation with "shapleySubsetMC" requires here 1000 samples of $(U^{(l)\prime})_{l\in [1:n]}$ and takes around 11 seconds. Remark that this time difference can become more accentuated if the function $f$ is a costly computer code. Finally, the estimator of the Shapley effects given by the linear approximation converges almost surely when $n$ goes to $+\infty$, whereas the estimator of the Shapley effects given by "shapleySubsetMC" is only shown to converge in probability when the sample size and $N_{tot}$ go to $+\infty$ (see \cite{broto_variance_2018}).\\

To conclude, we have provided a framework where the theoretical results of Section \ref{section_approx_gaussian_theore} can be applied. We have illustrated this framework with numerical experiments on generated data. We have showed that, in this framework, to estimate the Shapley effects, the Gaussian linear approximation provides an estimator much faster and much more accurate than the general estimator given by "shapleySubsetMC".

\section{Conclusion}\label{section_conclu}
In this paper, we worked on the Gaussian linear framework approximation to estimate the Shapley effects, in order to take advantage of the simplicity brought by this framework. First, we focused on the case where the inputs are Gaussian variables converging to their means. This setting is motivated, in particular, by the case of uncertainties on physical quantities that are reduced by taking more and more measurements. We showed that, to estimate the Shapley effects, one can replace the true model $f$ by three possible linear approximations: the exact Taylor polynomial approximation, a finite difference approximation and a linear regression. We gave the rate of convergence of the difference between the Shapley effects of the linear approximations and the Shapley effects of the true model. These results are illustrated by a simulated application that highlights the accuracy of the approximations.
Then, we focused on the case where the inputs are given by an empirical mean. In this case, we proved that the instinctive idea to replace the empirical mean by a Gaussian vector and the true model by a linear approximation around the mean indeed gives good approximations of the Shapley effects. We highlighted the benefits of these estimators on numerical experiments.

Several questions remain open to future work. In particular, it would be valuable to obtain more insight on the choice between the general estimator of the Shapley effects for non-linear models and the estimators based on Gaussian linear approximations. Quantitative criteria for this choice, based for instance on the magnitude of the input uncertainties or on the number of input samples that are available, would be beneficial. Regarding the results on the impact of individual estimation errors in Section \ref{section_approx_gaussian_appli}, it would be interesting to obtain extensions to estimators of quantities of interest that are not only empirical means, for instance general M-estimators.

\section*{Acknowledgements}
We acknowledge the financial support of the Cross- Disciplinary Program on Numerical Simulation of CEA, the French Alternative Energies and Atomic Energy Commission. We would like to thank BPI France for co-financing this work, as part of the PIA (Programme d'Investissements d'Avenir) - Grand D\'{e}fi du Num\'{e}rique 2, supporting the PROBANT project. We acknowledge the Institut de Mathématiques de Toulouse.

\bibliographystyle{alpha}
\bibliography{biblio}

\clearpage

\appendixpage

We will write $C_{\sup}$ for a generic non-negative finite constant. The actual value of $C_{\sup}$ is of no interest and can change in the same sequence of equations. Similarly, we will write $C_{\inf}$ for a generic strictly positive constant.
Moreover, for all $u\subset [1:p]$, if $Z$ is a random vector in $\R^p$ and $g$ is a function from $\R^p$ to $\R$ such that  $\E(g(Z)^2)<+\infty$ and $\V(g(Z))>0$, let $S_u^{cl}(Z,g)$ be the closed Sobol index (see \cite{gamboa_statistical_2016} for example) for the input vector $Z$ and the model $g$, defined by:
$$
S_u^{cl}(Z,g)=\frac{\V(\E(g(Z)|Z_u))}{\V(g(Z))}.
$$

\section{Proofs for Section \ref{section_approx_linear}}

\textbf{Proof of Proposition \ref{prop_linear1}}

We divide the proof into several lemmas. We assume that the assumptions of Proposition \ref{prop_linear1} hold throughout this proof.

Let $\varepsilon \in ]0,1[$ be such that $f$ is $\mathcal{C}^3$ on $\overline{B}(\mu,\varepsilon)$ and such that, for all $x\in \overline{B}(\mu,\varepsilon)$, we have $Df(x)\neq 0$. Since $\mu^{\{n\}}$ converges to $\mu$, there exists $N\in \N$ such that, for all $n\geq N$, $\mu^{\{n\}}\in B(\mu,\varepsilon\slash 2)$. In the following, we assume that $n$ is larger than $N$.

\begin{lm}\label{lm_f2_f3}
For all $x\in B(\mu^{\{n\}}, \varepsilon\slash 2)$, we have
$$
|R_1^{\{n\}}(x)|\leq C_1 \|x-\mu^{\{n\}}\|^2,\;\; |R_2^{\{n\}}(x)|\leq C_1' \|x-\mu^{\{n\}}\|^3
$$
and for all $x \notin B(\mu^{\{n\}}, \varepsilon\slash 2)$, 
$$
|R_1^{\{n\}}(x)|\leq C_2 \|x-\mu^{\{n\}}\|^k, \;\; |R_2^{\{n\}}(x)|\leq C_2' \|x-\mu^{\{n\}}\|^k,
$$
where $C_1,C_1',C_2$ and $C_2'$ are positive constants that do not depend on $n$.
\end{lm}

\begin{proof}
Using Taylor's theorem, for all $x\in B(\mu^{\{n\}}, \frac{\varepsilon}{2})$, there exist $\theta_2(n,x), \theta_3(n,x) \linebreak \in ]0,1[$ such that 
\begin{eqnarray*}
f(x) & = &f_0^{\{n\}}+f_1^{\{n\}}(x)+ \frac{1}{2}D^2f(\mu^{\{n\}}+\theta_2(n,x)(x-\mu^{\{n\}}))(x-\mu^{\{n\}})\\
&=& f_0^{\{n\}}+f_1^{\{n\}}(x)+f_2^{\{n\}}(x)\\
 & & + \frac{1}{6}D^3f(\mu^{\{n\}}+\theta_3(n,x)(x-\mu^{\{n\}}))(x-\mu^{\{n\}}).
\end{eqnarray*}
Let $C_1=\frac{1}{2}\max_{x \in \overline{B}(\mu,\varepsilon)}\| D^2f(x)\|$ and $C_1'=\frac{1}{6}\max_{x \in \overline{B}(\mu,\varepsilon)}\| D^3f(x)\|$, where $\|\cdot\|$ also means the operator norm of a multilinear form. Thus, for all $x\in B(\mu,\frac{\varepsilon}{2})$,
$$
|R_1^{\{n\}}(x)|\leq C_1 \|x-\mu^{\{n\}}\|^2,\;\;|R_2^{\{n\}}(x)|\leq C_1' \|x-\mu^{\{n\}}\|^3.
$$
Moreover, $f$ is subpolynomial, so $\exists k \geq 3$, and $C<+\infty$ such that, $\forall x \in \R^p$,
$$
|f(x)|\leq C(1+ \|x\|^k).
$$
Hence, taking $C'= C(2\|\mu\|+2)^k$, we have
$$
|f(x)|\leq   C(1+2^k\|x-\mu^{\{n\}}\|^k+2^k\|\mu^{\{n\}}\|^k)\leq C'(1+\|x-\mu^{\{n\}}\|^k).
$$
Hence, taking $C'':= C'+\max_{y\in \overline{B}(\mu,\varepsilon)}\|Df(y)\|$, we have
$$
|R_1^{\{n\}}(x)|\leq |f(x)|+ \max_{y\in \overline{B}(\mu,\varepsilon)}\|Df(y)\|\| x-\mu^{\{n\}}\| \leq C''(1+\|x-\mu^{\{n\}}\|^k).
$$
Now, taking $C_2:=C''\left(1+ (\frac{2}{\varepsilon})^k\right)$, we have, for all $x \notin B(\mu^{\{n\}}, \varepsilon\slash 2)$,
$$
|R_1^{\{n\}}(x)|\leq  C''+C''\|x-\mu^{\{n\}}\|^k \leq C_2 \| x - \mu^{\{n\}} \|^k.
$$
Similarly, there exists $C_2'<+\infty$ such that
$$
|R_2^{\{n\}}(x)|\leq C_2' \|x-\mu^{\{n\}}\|^k.
$$
\end{proof}

\begin{lm}\label{lm_cov_f1_f2}
We have
$$
 \cov(\E(f_1^{\{n\}}(X^{\{n\}})|X_u^{\{n\}}),f_2^{\{n\}}(X^{\{n\}})|X_u^{\{n\}}))=0.
$$
\end{lm}

\begin{proof}
Let $n\in \N$. To simplify notation, let $A=X^{\{n\}}-\mu^{\{n\}}$, $\beta \in \R^p$ be the vector of the linear application $Df(\mu^{\{n\}})$ and $\Gamma \in \mathcal{M}_p(\R)$ be symmetric the matrix of the quadratic form $\frac{1}{2}D^2f(\mu^{\{n\}})$. Then,
\begin{eqnarray*}
& & \cov(\E(f_1^{\{n\}}(X^{\{n\}})|X_u^{\{n\}}),\E(f_2^{\{n\}}(X^{\{n\}})|X_u^{\{n\}}))\\
&=& \cov(\E(\beta^T A)|A_u), \E(A^T \Gamma A |A_u))\\
&=& \E \left( \left[ \beta_u^T A_u+ \beta_{-u}^T\E(A_{-u}|A_u)\right] \left[A_u^T \Gamma_{u,u}A_u +2 A_u^T \Gamma_{u,-u} \E(A_{-u}|A_u)+ \E (A_{-u}^T \Gamma_{-u,-u} A_{-u}| A_u) \right] \right)\\
&=&\E \left( \left[ \beta_u^T A_u+ \beta_{-u}^T\E(A_{-u}|A_u)\right] \E (A_{-u}^T \Gamma_{-u,-u} A_{-u}| A_u) \right)
\end{eqnarray*}
since all the other terms are linear combinations of expectations of products of three zero-mean Gaussian variables. Indeed, the coefficients of $\E(A_{-u}|A_u)$ are linear combinations of the coefficients of $A_u$. Now,
\begin{eqnarray*}
\E \left( \beta_u^T A_u \times \E (A_{-u}^T \Gamma_{-u,-u} A_{-u}| A_u) \right)&=&\E \left( \E ( \beta_u^T A_u \times A_{-u}^T \Gamma_{-u,-u} A_{-u}| A_u) \right)\\
&=&  \E ( \beta_u^T A_u \times A_{-u}^T \Gamma_{-u,-u} A_{-u})\\
&=&0.
\end{eqnarray*}
Similarly, the term $\E \left( \beta_{-u}\E(A_{-u}|A_u) \E (A_{-u}^T \Gamma_{-u,-u} A_{-u}| A_u) \right)$ is equal to 0.
\end{proof}

\begin{lm}\label{lm_bound_var_cov}
There exists $C_{\sup}<+\infty$ such that, for all $u \subset [1:p]$,
$$
\V(\E(\sqrt{a^{\{n\}} }R_1^{\{n\}}(X^{\{n\}})|X^{\{n\}}_u))\leq\frac{C_{\sup}}{a^{\{n\}}},
$$
and
$$
\left|\cov(\E(\sqrt{a^{\{n\}} } f_1^{\{n\}}(X^{\{n\}})|X_u^{\{n\}}),\E(\sqrt{a^{\{n\}} } R_1^{\{n\}}(X^{\{n\}})|X_u^{\{n\}}))\right|\leq  \frac{C_{\sup}}{a^{\{n\}}}.
$$
\end{lm}

\begin{proof}
Using Lemma \ref{lm_f2_f3}, we have,
\begin{eqnarray*}
\E(|\sqrt{a^{\{n\}} }R_1^{\{n\}}(X^{\{n\}})|^2)& =& \E(|\sqrt{a^{\{n\}} }R_1^{\{n\}}(X^{\{n\}})|^2 \mathds{1}_{\|X_n\| < \frac{\varepsilon}{2}}) + \E(|\sqrt{a^{\{n\}} }R_1^{\{n\}} (X^{\{n\}})|^2 \mathds{1}_{\|X_n\| \geq \frac{\varepsilon}{2}})\\
& \leq &  \frac{C_1^2}{a^{\{n\}}}\E(\|\sqrt{a^{\{n\}} } ( X^{\{n\}}-\mu^{\{n\}})\|^{4})\\
& & +\frac{C_2^2}{a^{\{n\}(k-1) }}\E(\|\sqrt{a^{\{n\}} } ( X^{\{n\}}-\mu^{\{n\}})\|^{2k})\\
& \leq & \frac{C_{\sup}}{a^{\{n\}}},
\end{eqnarray*}
since $a^{\{n\}}\Sigma^{\{n\}}$ is bounded.
Hence,
$$
\V(\sqrt{a^{\{n\}}}R_1^{\{n\}}(X^{\{n\}}))\leq \frac{C_{\sup}}{a^{\{n\}}}.
$$
Moreover, for all $u \subset[1:p]$,
$$
0\leq \V(\E(a^{\{n\}}R_1^{\{n\}}(X^{\{n\}})|X^{\{n\}}_u))\leq \V(a^{\{n\}}R_1^{\{n\}}(X^{\{n\}}))\leq \frac{C_{\sup}}{a^{\{n\}}}.
$$
For all $u\subset [1:p]$,
\begin{eqnarray*}
& & \cov(\E(\sqrt{a^{\{n\}}}f_1^{\{n\}}(X^{\{n\}})|X_u^{\{n\}}),\E(\sqrt{a^{\{n\}}}R_1^{\{n\}}(X^{\{n\}})|X_u^{\{n\}}))\\
&=&\cov(\E(\sqrt{a^{\{n\}}}f_1^{\{n\}}(X^{\{n\}})|X_u^{\{n\}}),\E(\sqrt{a^{\{n\}}}f_2^{\{n\}}(X^{\{n\}})|X_u^{\{n\}}))\\
&& + \cov(\E(\sqrt{a^{\{n\}}}f_1^{\{n\}}(X^{\{n\}})|X_u^{\{n\}}),\E(\sqrt{a^{\{n\}}}R_2^{\{n\}}(X^{\{n\}})|X_u^{\{n\}}))\\
&=& \cov(\E(\sqrt{a^{\{n\}}}f_1^{\{n\}}(X^{\{n\}})|X_u^{\{n\}}),\E(\sqrt{a^{\{n\}}}R_2^{\{n\}}(X^{\{n\}})|X_u^{\{n\}})),
\end{eqnarray*}
using Lemma \ref{lm_cov_f1_f2}. Now, by Cauchy-Schwarz inequality,
\begin{eqnarray*}
& & \left| \cov(\E(\sqrt{a^{\{n\}}}f_1^{\{n\}}(X^{\{n\}})|X_u^{\{n\}}),\E(\sqrt{a^{\{n\}}}R_2^{\{n\}}(X^{\{n\}})|X_u^{\{n\}})) \right|\\
& \leq & \sqrt{\V (\sqrt{a^{\{n\}}}f_1^{\{n\}}(X^{\{n\}})|X_u^{\{n\}})}\sqrt{\V(\sqrt{a^{\{n\}}}R_2^{\{n\}}(X^{\{n\}})|X_u^{\{n\}})}\\
& \leq &  \sqrt{\V (\sqrt{a^{\{n\}}}f_1^{\{n\}}(X^{\{n\}}))}\sqrt{\V(\sqrt{a^{\{n\}}}R_2^{\{n\}}(X^{\{n\}}))}.
\end{eqnarray*}
Now, by Lemma \ref{lm_f2_f3}, we have,
\begin{eqnarray*}
& & \E(|\sqrt{a^{\{n\}}}R_2^{\{n\}}(X^{\{n\}})|^2)\\
& =& \E(|\sqrt{a^{\{n\}}}R_2^{\{n\}}(X^{\{n\}})|^2 \mathds{1}_{\|X_n\| \leq \frac{\varepsilon}{2}}) + \E(|\sqrt{a^{\{n\}}}R_2^{\{n\}} (X^{\{n\}})|^2 \mathds{1}_{\|X_n\| \geq \frac{\varepsilon}{2}})\\
& \leq &  \frac{C_1^2}{a^{\{n\}2}}\E(\|\sqrt{a^{\{n\}}} ( X^{\{n\}}-\mu^{\{n\}})\|^{6})\\
& & +\frac{C_2^2}{a^{\{n\}(k-1) }}\E(\|\sqrt{a^{\{n\}}} ( X^{\{n\}}-\mu^{\{n\}})\|^{k\times 2})\\
& \leq & \frac{C_{\sup}}{a^{\{n\}2}}.
\end{eqnarray*}
Furthermore,
\begin{eqnarray*}
\V (\sqrt{a^{\{n\}}}f_1^{\{n\}}(X^{\{n\}}))
& \leq &  \max_{x \in \overline{B}(\mu^{\{n\}},\varepsilon\slash 2)} \| Df(x)\|\E\left( \|\sqrt{a^{\{n\}}}(X^{\{n\}}-\mu^{\{n\}})\|\right)\\
& \leq & C_{\sup}.
\end{eqnarray*}
Finally,
$$
\left| \cov(\E(\sqrt{a^{\{n\}}}f_1^{\{n\}}(X^{\{n\}})|X_u^{\{n\}}),\E(\sqrt{a^{\{n\}}}R_1^{\{n\}}(X^{\{n\}})|X_u^{\{n\}})) \right| \leq \frac{C_{\sup}}{a^{\{n\}}},
$$
that concludes the proof of Lemma \ref{lm_bound_var_cov}.
\end{proof}

\begin{lm}\label{lm_Su1}
For all $u\subset [1:p]$,
$$
S_u^{cl}(X^{\{n\}},f)=S_u^{cl}(X^{\{n\}},f_1^{\{n\}})+O\left( \frac{1}{a^{\{n\}}}\right).
$$

\end{lm}

\begin{proof}
We have
$$
f(X^{\{n\}})=f(\mu^{\{n\}})+f_1^{\{n\}}(X^{\{n\}})+R_1^{\{n\}}(X^{\{n\}}).
$$
For all $u \subset[1:p]$, we have 
$$
\E(f(X^{\{n\}})|X_u^{\{n\}})=f(\mu^{\{n\}})+\E(f_1^{\{n\}}(X^{\{n\}})|X_u^{\{n\}})+\E(R_1^{\{n\}}(X^{\{n\}})|X_u^{\{n\}}),
$$
so,
\begin{eqnarray*}
& & a^{\{n\}}\V(\E(f(X^{\{n\}})|X_u^{\{n\}}))\\
&=& \V(\E(\sqrt{a^{\{n\}}} f_1^{\{n\}} (X^{\{n\}}) | X_u^{\{n\}})) + \V(\E(\sqrt{a^{\{n\}}}R_1^{\{n\}}(X^{\{n\}})|X_u^{\{n\}}))\\ & & + 2 \cov(\E(\sqrt{a^{\{n\}}}f_1^{\{n\}}(X^{\{n\}})|X_u^{\{n\}}),\E(\sqrt{a^{\{n\}}}R_1^{\{n\}}(X^{\{n\}})|X_u^{\{n\}}))\\
&=& \V(\E(\sqrt{a^{\{n\}}}f_1^{\{n\}}(X^{\{n\}})|X_u^{\{n\}}))+ O(\frac{1}{a^{\{n\}}}),
\end{eqnarray*}
by Lemma \ref{lm_bound_var_cov}.
Hence, for $u=[1:p]$, we have
$$
a^{\{n\}}\V(f(X^{\{n\}}))=\V(\sqrt{a^{\{n\}}}f_1^{\{n\}}(X^{\{n\}})) + O(\frac{1}{a^{\{n\}}}).
$$

Thus, for all $u\subset [1:p]$,
\begin{eqnarray*}
S_u^{cl}(X^{\{n\}},f)&=&\frac{\V(\E(f(X^{\{n\}})|X_u^{\{n\}}))}{\V(f(X^{\{n\}}))}\\
&=&\frac{a^{\{n\}}\V(\E(f(X^{\{n\}})|X_u^{\{n\}}))}{a^{\{n\}}\V(f(X^{\{n\}}))}\\
& =& \frac{\V(\E( \sqrt{a^{\{n\}}}f_1^{\{n\}}(X^{\{n\}})|X_u^{\{n\}}))+O(\frac{1}{a^{\{n\}}})}{\V( \sqrt{a^{\{n\}}} f_1^{\{n\}}(X^{\{n\}})) + O(\frac{1}{a^{\{n\}}})}\\
&=&\frac{\V(\sqrt{a^{\{n\}}}f_1^{\{n\}}(X^{\{n\}})|X_u^{\{n\}})}{\V(\sqrt{a^{\{n\}}} f_1^{\{n\}}(X^{\{n\}}))}+O(\frac{1}{a^{\{n\}}})\\
&=&S_u^{cl}(X^{\{n\}},f_1^{\{n\}})+O\left( \frac{1}{a^{\{n\}}}\right),
\end{eqnarray*}
where we used that, 
\begin{eqnarray*}
\V(\sqrt{a^{\{n\}}} f_1^{\{n\}}(X^{\{n\}}))&= & Df(\mu^{\{n\}})(a^{\{n\}}\Sigma^{\{n\}})Df(\mu^{\{n\}})^T \\
& \geq & \lambda_{\min}(a^{\{n\}}\Sigma^{\{n\}})\inf_{x\in \overline{B}(\mu,\varepsilon\slash 2)}\|Df(x)\|^2\\
&\geq & C_{inf}.
\end{eqnarray*}

\end{proof}

Now we have proved the convergence of the closed Sobol indices, we can prove Proposition \ref{prop_linear1} easily.

\begin{proof}
By Lemma \ref{lm_Su1} and applying the linearity of the Shapley effects with respect to the Sobol indices, we have
$$
\eta(X^{\{n\}},f)=\eta(X^{\{n\}},f_1^{\{n\}})+O(\frac{1}{a^{\{n\}}}).
$$
\end{proof}

\textbf{Proof of Remark \ref{rmk}}
\begin{proof}
Let  $X^{\{n\}}=(X^{\{n\}}_1,X^{\{n\}}_2)\sim \mathcal{N}(0,\frac{1}{a^{\{n\}}} I_2)$ and $Y^{\{n\}}=f(X^{\{n\}})=X_1^{\{n\}}+X_2^{\{n\}2}$, we have $f_1^{\{n\}}(X^{\{n\}})=X_1^{\{n\}}$ and $R_1^{\{n\}}(X^{\{n\}})=X_2^{\{n\}2}$. Thus, $\eta_1(X^{\{n\}},f_1^{\{n\}}) =1$ and $\eta_2(X^{\{n\}},f_1^{\{n\}})=0$.
Now, let us compute the Shapley effects $\eta(X^{\{n\}},f)$. We have
\begin{eqnarray*}
\V(f(X^{\{n\}}))&=&\V(X_1^{\{n\}})+\V(X_2^{\{n\}2})\\
&=&\V(X_1^{\{n\}})+\E(X_2^{\{n\}4})- \E(X_2^{\{n\}2})^2\\
&=&\frac{1}{a^{\{n\}}}+ \frac{3}{a^{\{n\}2}}-\frac{1}{a^{\{n\}2}}\\
&=&\frac{a^{\{n\}}+2}{a^{\{n\}2}}.
\end{eqnarray*}
Moreover,
$$
\V(\E(f(X^{\{n\}})|X_1^{\{n\}}))=\V(X_1^{\{n\}}+ \frac{1}{a^{\{n\}}})=\V(X_1^{\{n\}})=\frac{1}{a^{\{n\}}}
$$
and 
$$
\V(\E(f(X^{\{n\}})|X_2^{\{n\}}))=\V(X_2^{\{n\}2})=\E(X_2^{\{n\}4})- \E(X_2^{\{n\}2})^2=\frac{3-1}{a^{\{n\}2}}= \frac{2}{a^{\{n\}2}}.
$$
Hence,
\begin{eqnarray*}
\eta_1(X^{\{n\}},f)&=& \frac{a^{\{n\}2}}{(a^{\{n\}}+2) 2}\left( \frac{1}{a^{\{n\}}}+ \frac{a^{\{n\}}+2}{a^{\{n\}2}} - \frac{2}{a^{\{n\}2}}\right)= \frac{a^{\{n\}}}{a^{\{n\}}+2},
\end{eqnarray*}
and
$$
\eta_2(X^{\{n\}},f)=\frac{2}{a^{\{n\}}+2}.
$$
\end{proof}

\textbf{Proof of Proposition \ref{prop_linear2}}

As in the proof of Proposition \ref{prop_linear1}, we first prove the convergence for the closed Sobol indices. To simplify notation, let $\Gamma^{\{n\}}:=a^{\{n\}}\Sigma^{\{n\}}$.

\begin{lm}\label{lm_linear_prop2}
Under the assumptions of Proposition \ref{prop_linear2}, for all $u\subset[1:p]$, we have $$ 
S_u^{cl}(f_1^{\{n\}}(X^{\{n\}}))=S_u^{cl}(f_1(X^*))+O(\| \mu^{\{n\}}-\mu\|)+O(\|\Gamma^{\{n\}}-\Sigma\| ).
$$
\end{lm}

\begin{proof}
We have
\begin{eqnarray*}
& & \V(\sqrt{a^{\{n\}}} f_1^{\{n\}}(X^{\{n\}}))-\V(f_1(X^*))\\
&=&Df(\mu^{\{n\}})\Gamma^{\{n\}} Df(\mu^{\{n\}})^T-Df(\mu)\Sigma Df(\mu )^T \\
&=& Df(\mu^{\{n\}})\Gamma^{\{n\}} \left[Df(\mu^{\{n\}})^T-Df(\mu)^T\right]+ Df(\mu^{\{n\}})\left[\Gamma^{\{n\}}- \Sigma\right] Df(\mu)^T   \\
 & &  \left[Df(\mu^{\{n\}})-Df(\mu)\right] \Sigma Df(\mu) ^T \\
&=& O(\| Df(\mu^{\{n\}})-Df(\mu)\|)+O(\|\Gamma^{\{n\}}-\Sigma\| )\\
&=&  O(\| \mu^{\{n\}}-\mu\|)+O(\|\Gamma^{\{n\}}-\Sigma\| ),
\end{eqnarray*}
using that $Df$ is Lipschitz continuous on a neighbourhood of $\mu$ (thanks to the continuity of $D^2f$).

Moreover, for all $\emptyset \varsubsetneq u\varsubsetneq[1:p]$, we have
\begin{eqnarray*}
& & \V(\E(\sqrt{a^{\{n\}}} f_1^{\{n\}}(X^{\{n\}})|X_u^{\{n\}}))-  \V(\E(f_1(X^*)|X^*_u))\\
&=&\V(\sqrt{a^{\{n\}}}f_1^{\{n\}}(X^{\{n\}}))-\E(\V(\sqrt{a^{\{n\}}}f_1(X^{\{n\}})|X_{-u}^{\{n\}})) -  \V(f_1(X^*))+ \E(\V(f_1(X^*)|X^*_u)) \\
&=&Df(\mu^{\{n\}})\Gamma^{\{n\}} Df(\mu^{\{n\}})^T - Df(\mu^{\{n\}})_{u}(\Gamma^{\{n\}}_{u,u}-\Gamma^{\{n\}}_{u,-u}\Gamma^{\{n\}-1}_{-u,-u}\Gamma^{\{n\}}_{-u,u}) Df(\mu^{\{n\}})_u^T\\
& &-  Df(\mu)\Sigma Df(\mu)^T - Df(\mu)_{u}(\Sigma_{u,u}-\Sigma_{u,-u}\Sigma_{-u,-u}^{-1}\Sigma_{-u,u}) Df(\mu)_u^T\\
&=& O(\| \mu^{\{n\}}-\mu\|)+O(\|\Gamma^{\{n\}}-\Sigma\| ),
\end{eqnarray*}
proceeding as previously and using the fact that the operator norm of a submatrix is smaller than the operator norm of the whole matrix.

Hence,
$$
S_u^{cl}(X^{\{n\}},f_1^{\{n\}})=S_u^{cl}(X^*,f_1)+O(\| \mu^{\{n\}}-\mu\|)+O(\|\Gamma^{\{n\}}-\Sigma\| ).
$$

\end{proof}

Now, we can easily prove Proposition \ref{prop_linear2}.
\begin{proof}
By Lemma \ref{lm_linear_prop2} and applying the linearity of the Shapley effects with respect to the Sobol indices, we have
$$
\eta(f_1^{\{n\}}(X^{\{n\}}))=\eta(f_1(X^*))+O(\| \mu^{\{n\}}-\mu\|)+O(\|\Gamma^{\{n\}}-\Sigma\| ).
$$
\end{proof}

\textbf{Proof of Proposition \ref{prop_linear_Dh}}

Under the assumption of Proposition \ref{prop_linear_Dh}, let $\varepsilon >0$ be such that $f$ is $\mathcal{C}^3$ on $\overline{B}(\mu,\varepsilon)$ and such that, for all $x\in \overline{B}(\mu,\varepsilon)$, we have $Df(x)\neq 0$. Since $\mu^{\{n\}}$ converges to $\mu$, there exists $N\in \N$ such that, for all $n\geq N$, $\mu^{\{n\}}\in B(\mu,\varepsilon\slash 2)$. In the following, we assume that $n$ is larger than $N$.

\begin{lm}\label{lm_linear_Dh1}
For all $x\in \overline{B}(\mu,\frac{\varepsilon}{2})$ and $h\in (\R_+^*)^p$ such that $\|h\| \leq \frac{\varepsilon}{2}$, we have
$$
\| \widehat{D}_h f(x)-Df(x)\| \leq \frac{1}{6}\max_{i\in [1:p]} \max_{y \in \overline{B}(\mu,\varepsilon)}|\partial_i^3 f(y)| \|h\|^2
$$
\end{lm}

\begin{proof}
Let $x\in \overline{B}(\mu,\frac{\varepsilon}{2})$ and $h\in (\R_+^*)^p$ such that $\|h\| \leq \frac{\varepsilon}{2}$. For all $i\in [1:p]$, using Taylor's theorem, there exist $\theta_{x,h,i}^+,\theta_{x,h,i}^- \in ]0,1[$ such that 
$$
\frac{f\left( x+e_i h_i\right)-f\left( x-e_i h_i\right)}{2h_i}=\partial_if(x)+\frac{h_i^2}{12}\left( \partial_i^3 f(x+\theta_{x,h,i}^+h)+ \partial_i^3 f(x-\theta_{x,h,i}^-h)\right).
$$
Hence,
\begin{eqnarray*}
\| \widehat{D}_h f(x)-Df(x)\| & \leq & 
\sum_{i=1}^p \left| \left[\widehat{D}_h f(x)-Df(x)\right]_i \right| \\ &\leq &  \frac{1}{6} \max_{i\in [1:p]} \max_{y \in \overline{B}(\mu,\varepsilon)}| \partial_i^3 f(y)| \sum_{i=1}^p h_i^2\\
 & = &  \frac{1}{6}\max_{i\in [1:p]} \max_{y \in \overline{B}(\mu,\varepsilon)}|\partial_i^3 f(y)| \|h\|^2.
 \end{eqnarray*}
\end{proof}

\begin{lm}\label{lm_linear_Dh2}
For all linear functions $l_1$ and $l_2$ from $\R^p$ to $\R$, we have
$$
\left| \V(\E(l_1(X^{\{n\}})|X_u^{\{n\}})-
\V(\E(l_2(X^{\{n\}})|X_u^{\{n\}})\right| \leq \frac{C_{\sup}}{a^{\{n\}}}  \|l_1-l_2 \|.
$$
\end{lm}

\begin{proof}
For all $u \subset[1:p]$, let $\phi_u^{\{n\}}:\R^{|u|}\longrightarrow \R^p$ be defined by
$$
\phi_u^{\{n\}}(x_u)=\begin{pmatrix}
x_u \\ \mu_{-u}^{\{n\}}+\Gamma^{\{n\}}_{-u,u}\Gamma_{u,u}^{\{n\}-1}(x_u-\mu_u^{\{n\}})\end{pmatrix}
$$
and $\phi_{[1:p]}^{\{n\}}=id_{\R^p}$.

Let $u \subset [1:p]$. Then
$$
\E(X^{\{n\}}|X_u^{\{n\}})=\phi_u^{\{n\}}(X_u^{\{n\}}).
$$
Now, for all linear function $l:\R^p\longrightarrow \R$, we have
$$
\E(l(X^{\{n\}})|X_u^{\{n\}})=l\left( \E(X^{\{n\}}|X_u^{\{n\}}) \right) = l(\phi_u^{\{n\}}(X_u^{\{n\}})),
$$
so, identifying a linear function from $\R^p$ to $\R$ with its matrix of size $1\times p$, we have
$$
\V\left(\E(l(X^{\{n\}})|X_u^{\{n\}})\right)= l \phi_u^{\{n\}} \frac{\Gamma_{u,u}^{\{n\}}}{a^{\{n\}}} \phi_u^{\{n\}T} l^T.
$$
Hence, for $l=l_1$ and $l=l_2$, one can show that,
$$
\left| \V(\E(l_1(X^{\{n\}})|X_u^{\{n\}}))-
\V(\E(l_2(X^{\{n\}})|X_u^{\{n\}}))\right| \leq \frac{C_{\sup}}{a^{\{n\}}}  \|l_1-l_2 \|.
$$

\end{proof}

Now, we can prove Proposition \ref{prop_linear_Dh}.
\begin{proof}
By Lemmas \ref{lm_linear_Dh1} and \ref{lm_linear_Dh2}, we have, for all $u\subset[1:p]$, 
$$
 \V(\E(\sqrt{a^{\{n\}}}f_{1}^{\{n\}}(X^{\{n\}})|X_u^{\{n\}})-
\V(\E(\sqrt{a^{\{n\}}}\tilde{f}_{1,h^{\{n\}}}^{\{n\}}(X^{\{n\}})|X_u^{\{n\}}) =O\left(\|h^{\{n\}}\|^2\right).
$$
Thus,
$$
S_u^{cl}(X^{\{n\}},f_{1}^{\{n\}})-S_u^{cl}(X^{\{n\}},\tilde{f}_{1,h^{\{n\}}}^{\{n\}})=O\left( \|h^{\{n\}}\|^2\right),
$$ 
so
$$
\eta(X^{\{n\}},f_1^{\{n\}})-\eta(X^{\{n\}},\tilde{f}_{1,h^{\{n\}}}^{\{n\}})=O\left( \|h^{\{n\}}\|^2\right).
$$ 

\end{proof}

\textbf{Proof of Proposition \ref{prop_linear_regression}}

Under the assumption of Proposition \ref{prop_linear_Dh}, let $\varepsilon >0$ be such that $f$ is $\mathcal{C}^3$ on $\overline{B}(\mu,\varepsilon)$ and such that, for all $x\in \overline{B}(\mu,\varepsilon)$, we have $Df(x)\neq 0$. Since $\mu^{\{n\}}$ converges to $\mu$, there exists $N\in \N$ such that, for all $n\geq N$, $\mu^{\{n\}}\in B(\mu,\varepsilon\slash 2)$. In the following, we assume that $n$ is larger than $N$.

\begin{lm}\label{lm_linear_regression1}
There exists $C_{\sup}$ such that, with probability at least \linebreak $ 1 -2 p^2\exp(-C_{\inf} N )-4p\exp(-C_{\inf} N^2 )$,
$$
\left\|  \left( A^{\{n\}T}  A^{\{n\}}\right)^{-1}  A^{\{n\}T}\right\|\leq C_{\sup} \frac{\sqrt{a^{\{n\}}}}{\sqrt{N}}.
$$
\end{lm}

\begin{proof}
\begin{eqnarray*}
\left\|  \left( A^{\{n\}T}  A^{\{n\}}\right)^{-1}  A^{\{n\}T}\right\|^2  & = & \lambda_{\max}\left[  \left(  A^{\{n\}T}  A^{\{n\}}\right)^{-1}\right]\\
& =& \frac{a^{\{n\}}}{N} \lambda_{\max}\left[  \left( \frac{a^{\{n\}}}{N} A^{\{n\}T}  A^{\{n\}}\right)^{-1}\right].
\end{eqnarray*}
Now, by the strong law of large numbers, we have almost surely
\begin{eqnarray*}
\frac{a^{\{n\}}}{N} A^{\{n\}T}  A^{\{n\}}- (a^{\{n\}}-1) \begin{pmatrix}
1 \\ \mu^{\{n\}}
\end{pmatrix}\begin{pmatrix}
1 \\  \mu^{\{n\}}
\end{pmatrix}^T\\ \underset{N \to +\infty}{\longrightarrow} M_1^{\{n\}}:=
\begin{pmatrix}
1 & \mu^{\{n\}T}\\
\mu^{\{n\}} & \Gamma^{\{n\}}+\mu^{\{n\}}\mu^{\{n\}T}
\end{pmatrix}.
\end{eqnarray*}
Let $M_2^{\{n\}}:=\begin{pmatrix}
1 & \mu^{\{n\}T}\\
\mu^{\{n\}} & \lambda_{\inf} I_p+\mu^{\{n\}}\mu^{\{n\}T}
\end{pmatrix}$ and  $M_2:=\begin{pmatrix}
1 & \mu^{T}\\
\mu & \lambda_{\inf} I_p+\mu\mu^{T}
\end{pmatrix}$, where $\lambda_{\inf}>0$ is a lower-bound of the eigenvalues of $(\Gamma^{\{n\}})_n$. We can see that
$$
M_1^{\{n\}}\geq M_2^{\{n\}} \underset{n \to +\infty}{\longrightarrow} M_2.
$$
Now,
$$
\det(M_2)=\det(1)\det\left( [\lambda_{\inf}I_p+\mu \mu^T]- \mu 1^{-1} \mu^{T} \right)=\lambda_{\inf}^p>0.
$$
Hence, writing $\lambda_{\inf}'>0$ the smallest eigenvalue of $M_2$, we have that the eigenvalues of $M_1^{\{n\}}$ are lower-bounded by $\lambda_{\inf}'\slash 2$ for $n$ large enough.

Similarly, let $$M_3^{\{n\}}:=\begin{pmatrix}
1 & \mu^{\{n\}T}\\
\mu^{\{n\}} & \lambda_{\sup} I_p+\mu^{\{n\}}\mu^{\{n\}T}
\end{pmatrix}, \text{ and } M_3:=\begin{pmatrix}
1 & \mu^{T}\\
\mu & \lambda_{\sup} I_p+\mu\mu^{T}
\end{pmatrix},$$ where $\lambda_{\sup}>0$ is an upper-bound of the eigenvalues of $(\Gamma^{\{n\}})_n$. Writing $\lambda_{\sup}'<+\infty$ the largest eigenvalue of $M_3$, we have that the eigenvalues of $M_1^{\{n\}}$ are upper-bounded by $2\lambda_{\sup}'$ for $n$ large enough.

Now, since the eigenvalues of $(M_1^{\{n\}})_{n}$ are lower-bounded and upper-bounded, there exists $\alpha>0$ such that, for all $n\in \N$ (large enough), $\forall M \in S_p(\R)$, 
$$
\| M - M_1^{\{n\}} \| \leq \alpha \Longrightarrow | \lambda_{\min}(M) - \lambda_{\min}(M_1^{\{n\}}) | \leq \frac{\lambda_{\inf}'}{4}.
$$
Now, by Bernstein inequality,
\begin{eqnarray*}
& & \PP\left( \left\| \frac{a^{\{n\}}}{N} A^{\{n\}T}  A^{\{n\}}- (a^{\{n\}}-1) \begin{pmatrix}
1 \\ \mu^{\{n\}}
\end{pmatrix}\begin{pmatrix}
1 \\  \mu^{\{n\}}
\end{pmatrix}^T- M_1^{\{n\}}\right\| \leq \alpha \right) \\
& \geq &   1 -2 p^2\exp(-C_{\inf} N )- 2 \times 2 p\exp(-C_{\inf} N^2 )\\
& \geq &  1 - C_{\sup}\exp(-C_{\inf}N), 
\end{eqnarray*}
where the term $2 p^2\exp(-C_{\inf} N )$ bounds the difference of the submatrices of index $[2:p+1]\times [2:p+1]$ and the term $2 \times 2 p\exp(-C_{\inf} N^2 )$ bounds the differences of the submatrices of index $\{1\}\times [2:p+1]$ and $[2:p+1]\times \{1\}$.

Hence, with probability at least $ 1 - C_{\sup}\exp(-C_{\inf}N)$, we have
$$
\lambda_{\min}\left(\frac{a^{\{n\}}}{N} A^{\{n\}T}  A^{\{n\}}- (a^{\{n\}}-1) \begin{pmatrix}
1 \\ \mu^{\{n\}}
\end{pmatrix}\begin{pmatrix}
1 \\  \mu^{\{n\}}
\end{pmatrix}^T\right)\geq \frac{\lambda_{\inf}'}{4},
$$
and so
$$
\lambda_{\min}\left(\frac{a^{\{n\}}}{N} A^{\{n\}T}  A^{\{n\}}\right)\geq \frac{\lambda_{\inf}'}{4}.
$$

\end{proof}

\begin{lm}\label{lm_linear_regression2}
With probability at least $1 - C_{\sup}\exp(-C_{\inf}N)$, we have
$$
\left\|\widehat{\beta}^{\{n\}}-\nabla f (\mu^{\{n\}})\right\|\leq C_{\sup}\frac{1}{\sqrt{a^{\{n\}}}}.
$$
\end{lm}

\begin{proof}
Let $Z^{\{n\}}\sim \mathcal{N}(0,\Gamma^{\{n\}})$.
Then $\| X^{\{n\}} - \mu^{\{n\}} \| \leq \frac{\varepsilon}{2}$ with probability $
\PP (\|Z^{\{n\}}\| \leq  \frac{a^{\{n\}} \varepsilon}{2})\underset{n\to +\infty}{\longrightarrow}1$.
Let $\Omega_N^{\{n\}}:=\{\omega \in \Omega\; |\;\forall j\in [1:N]$, $\| X^{\{n\}(j)}(\omega) - \mu^{\{n\}} \| \leq \frac{\varepsilon}{2}\}$. Hence, 
$$
\PP(\Omega_N^{\{n\}}) \geq 1 - 2N\exp \left(- C_{\inf} a^{\{n\}} \right) \underset{n\to +\infty}{\longrightarrow}1.
$$

On $\overline{B}(\mu^{\{n\}},\frac{\varepsilon}{2})$, we have $f=f(\mu^{\{n\}})+f_1^{\{n\}}+R_1^{\{n\}}$. Hence, on $\Omega_N^{\{n\}}$, for all $j \in [1:N]$,
$$
f(X^{\{n\}(j)})=f(\mu^{\{n\}})+f_1^{\{n\}}(X^{\{n\}(j)})+R_1^{\{n\}}(X^{\{n\}(j)}).
$$
Thus,
$$
\widehat{\beta}^{\{n\}}=\left(A^{\{n\}T}A^{\{n\}}\right)^{-1} A^{\{n\}T} \bigg(f(\mu^{\{n\}})+f_1^{\{n\}}(X^{\{n\}(j)})+R_1^{\{n\}}(X^{\{n\}(j)})\bigg)_{j\in [1:N]}.
$$
Since $f(\mu^{\{n\}})+ f_1^{\{n\}}$ is a linear function with gradient vector $ \nabla f(\mu^{\{n\}})$ and with value at zero $f(\mu^{\{n\}})-Df(\mu^{\{n\}})\mu^{\{n\}}$, we have,
$$
\left(A^{\{n\}T}A^{\{n\}}\right)^{-1} A^{\{n\}T} (f(\mu^{\{n\}})+f_1^{\{n\}}(X^{\{n\}(j)}))_{j\in [1:N]}=\begin{pmatrix}  f(\mu^{\{n\}})-Df(\mu^{\{n\}})\mu^{\{n\}}\\ \nabla f(\mu^{\{n\}}) \end{pmatrix}.
$$
Hence, it remains to see if 
$$
\left(A^{\{n\}T}A^{\{n\}}\right)^{-1} A^{\{n\}T} (R_1^{\{n\}}(X^{\{n\}(j)}))_{j\in [1:N]}
$$
is small enough. By Lemma \ref{lm_f2_f3}, we have on $\Omega_N^{\{n\}}$,
\begin{eqnarray*}
\|(R_1^{\{n\}}(X^{\{n\}(j)}))_{j\in [1:N]}\|^2&=&\sum_{j=1}^N  R_1^{\{n\}}(X^{\{n\}(j)})^2\\
&\leq & C_{\sup} \sum_{j=1}^N \|X^{\{n\}(j)} - \mu^{\{n\}}\|^4\\
& \leq & \frac{C_{\sup}}{a^{\{n\}2}}\sum_{j=1}^N \|\sqrt{a^{\{n\}}}(X^{\{n\}(j)} - \mu^{\{n\}})\|^4.
\end{eqnarray*}
Hence, on $\Omega_N^{\{n\}}$,
$$
\|(R_1^{\{n\}}(X^{\{n\}(j)}))_{j\in [1:N]}\| \leq C_{\sup}  \frac{\sqrt{N}}{a^{\{n\}}}.
$$
Thus,
\begin{eqnarray*}
& & \left\|  \left(A^{\{n\}T}A^{\{n\}}\right)^{-1} A^{\{n\}T} (R_1^{\{n\}}(X^{\{n\}(j)}))_{j\in [1:N]} \right\|\\
& \leq & \left\|  \left(A^{\{n\}T}A^{\{n\}}\right)^{-1} A^{\{n\}T}\right\| \left\| (R_1^{\{n\}}(X^{\{n\}(j)}))_{j\in [1:N]} \right\|\\
& \leq & C_{\sup} \frac{1}{\sqrt{a^{\{n\}}}},
\end{eqnarray*}
with probability at least $1 - C_{\sup}\exp(-C_{\inf}N)$.
\end{proof}

Now, it is easy to prove Proposition \ref{prop_linear_regression}.
\begin{proof}
By Lemma \ref{lm_linear_Dh2} for $l_1=\widehat{\beta}^{\{n\}T}$ and $l_2=Df(\mu^{\{n\}})$, and by Lemma \ref{lm_linear_regression2} we have, with probability at least $1 - C_{\sup}\exp(-C_{\inf}N)$,
\begin{eqnarray*}
& & \left| \V(\E(\sqrt{a^{\{n\}}}Df(\mu^{\{n\}})X^{\{n\}}|X_u^{\{n\}}))-
\V(\E(\sqrt{a^{\{n\}}}\widehat{\beta}^{\{n\}T}X^{\{n\}}|X_u^{\{n\}}))\right| \\
&\leq& C_{\sup}  \|Df(\mu^{\{n\}})-\widehat{\beta}^{\{n\}T} \|\\
& \leq &C_{\sup}\frac{1}{\sqrt{a^{\{n\}}}},
\end{eqnarray*}
where the conditional expectations and the variances are conditional to $(X^{\{n\}(j)})_{j\in [1:N]}$. Thus, with probability at least $1 - C_{\sup}\exp(-C_{\inf}N)$, there exists $C_{\inf}>0$ such that, for $n$ large enough $\|\widehat{\beta}^{\{n\}T}\|\geq C_{\inf}$, thus $\V(\sqrt{a^{\{n\}}}\widehat{\beta}^{\{n\}^T}X^{\{n\}})$ is lower-bounded.
Hence, with probability at least $1 - C_{\sup}\exp(-C_{\inf}N)$,
$$
\left| S_u^{cl}(X^{\{n\}},f_1^{\{n\}})-S_u^{cl}(X^{\{n\}},\widehat{\beta}^{\{n\}T})\right|\leq C_{\sup}\frac{1}{\sqrt{a^{\{n\}}}},
$$ 
and so
$$
\left\|\eta(X^{\{n\}},f_1^{\{n\}})-\eta(X^{\{n\}},\widehat{\beta}^{\{n\}T})\right\|\leq C_{\sup}\frac{1}{\sqrt{a^{\{n\}}}}.
$$ 
\end{proof}

\section{Proofs for Section \ref{section_approx_gaussian}}
In this section, we prove Proposition \ref{prop_approx_gaussien} in Subsections B.1 to B.6 and we prove Corollary \ref{corol_approx_gaussien} in Subsection B.7.

\subsection{Introduction to the proof of Proposition \ref{prop_approx_gaussien}}
Recall that $(U^{(l)})_{l\in [1:n]}$ is an i.i.d. sample of $U$ with $\E(U)=\mu$ and $\V(U)=\Sigma$ and 
$$
\widehat{X}^{\{n\}}=\frac{1}{n}\sum_{l=1}^{n} U^{(l)}.
$$
Let $X^{\{n\}}\sim \mathcal{N}(\mu, \frac{1}{n}\Sigma)$. 
By Proposition \ref{prop_linear1}, we have 
$$  
\eta\left(  X^{\{n\}}, f\right) = \eta\left(  X^{\{n\}}, Df(\mu)\right)+ O\left(\frac{1}{a^{\{n\}}}\right)=\eta\left(  X^*, Df(\mu)\right)+ O\left(\frac{1}{a^{\{n\}}}\right).
$$
Hence, it remains to prove that
$$
\left\| \eta\left(  \widehat{X}^{\{n\}}, f\right) -\eta\left(  X^{\{n\}}, f\right)   \right\| \underset{n \to +\infty}{\longrightarrow}0,
$$
that is, writing $f_n:=\sqrt{n}\left(f\left(\frac{\cdot}{\sqrt{n}}+\mu\right)-f(\mu) \right)$ and $\tilde{X}^{\{n\}}:=\sqrt{n}(\widehat{X}^{\{n\}}- \mu)$, that
$$
\left\| \eta\left(\tilde{X}^{\{n\}}, f_n\right)- \eta\left(X^*, f_n  \right)   \right\|  \underset{n \to +\infty}{\longrightarrow}0.
$$

In Section \ref{section_lm_fn}, we give some lemmas of $f_n$. Then, defining
$$
E_{u,n,K}(Z):=\E\left(\E\left[f_n(Z)^2\mathds{1}_{\|Z\|_{\infty}\leq K}\middle|Z_u\right]^2\right),
$$
$$
E_{u,n}(Z):=\E\left(\E\left[f_n(Z)^2\middle|Z_u\right]^2\right),
$$
we prove in Section \ref{section_terme1} that $\sup_n |E_{u,n,K}(\tilde{X}^{\{n\}})-E_{u,n}(\tilde{X}^{\{n\}})|$ converges to $0$ when $K \to +\infty$. In particular, for $U\sim\mathcal{N}(\mu, \Sigma)$, the result holds for $\tilde{X}^{\{n\}}=X^*$.

Hence, for any $\varepsilon>0$, choosing $K$ such that $|E_{u,n,K}(\tilde{X}^{\{n\}})-E_{u,n}(\tilde{X}^{\{n\}})|<\varepsilon\slash 3$ and $|E_{u,n,K}(X^*)-E_{u,n}(X^*)|< \varepsilon \slash 3$, we show in Section \ref{section_terme2} that
$$
|E_{u,n,K}(X^*)-E_{u,n,K}(\tilde{X}^{\{n\}})|\underset{n\to +\infty}{\longrightarrow}0.
$$
In Section \ref{section_term3}, we conclude the proof that
$$
\left| \V(\E(f_n(\tilde{X}^{\{n\}})|\tilde{X}^{\{n\}}_u))-\V(\E(f_n(X^*)|X^*_u))  \right|\underset{n\to +\infty}{\longrightarrow}0.
$$
In Section \ref{section_approx_gaussien_conclu}, we conclude the proof that
$$
\left\| \eta\left(\tilde{X}^{\{n\}}, f_n\right)- \eta\left(X^*, f_n  \right)   \right\|  \underset{n \to +\infty}{\longrightarrow}0.
$$

The key of the proof is that the probability density function of $\tilde{X}^{\{n\}}$ converges uniformly to the one of $X^*$ by local limit theorem (see \cite{shervashidze1971uniform} or Theorem 19.1 of \cite{bhattacharya1986normal}).

\subsection{Part 1}\label{section_lm_fn}

\begin{lm}\label{lm_gaussian_approx1}
There exists $C_{\sup}<+\infty$ such that, for all $x \in \R^p$,
$$
|f_n(x)| \leq C_{\sup}\left(\|x\|\mathds{1}_{\|x\|\leq \sqrt{n}}+ \frac{\|x\|^k}{\sqrt{n}^{k-1}}\mathds{1}_{\|x\|> \sqrt{n}})\right),
$$
where we recall that $k\in \N^*$ is such that for all $x\in \R^p$, we have $|f(x)|\leq C(1+\|x\|^k)$.
\end{lm}
\begin{proof}
For all $x \in \R^p$, we have
\begin{eqnarray*}
\left|f \left(\frac{x}{\sqrt{n}}+\mu \right) - f \left(\mu \right) \right| & \leq & \left|f \left(\frac{x}{\sqrt{n}}+\mu \right) \right| + \left| f \left(\mu \right) \right| \\
& \leq & C_{\sup} \left( 1+ \left\| \frac{x}{\sqrt{n}}+\mu\right\|^k\right)+ |  f \left(\mu \right)|\\
& \leq & C_{\sup} \left( 1+ \left\| \frac{x}{\sqrt{n}}\right\|^k \right).
\end{eqnarray*}
Thus, for all $\|x\| \geq \sqrt{n}$, we have
$$
|f_n(x)|\leq C_{\sup} \frac{\| x \|^k}{\sqrt{n}^{k-1}}.
$$
If  $\|x \| \leq \sqrt{n}$, we have
\begin{eqnarray*}
\left|f \left(\frac{x}{\sqrt{n}}+\mu \right) - f \left(\mu \right) \right| & \leq &  \max_{ \|y\|\leq 1+\|\mu\|}\|Df(y)\| \left\| \frac{x}{\sqrt{n}}+\mu - \mu \right\|\\
& \leq & C_{\sup} \left\| \frac{x}{\sqrt{n}}\right\|,
\end{eqnarray*}
and thus,
$$
|f_n(x)| \leq C_{\sup} \|x\|.
$$
\end{proof}
In particular, 
$$
|f_n(x)| \leq C_{\sup} (\|x\|+ \|x\|^k),\;\;\; f_n(x)^2 \leq C_{\sup} (\|x\|^2+ \|x\|^{2k})
$$

\begin{lm}\label{lm_approx_gaussian_unif_bound}
For $i=1,2$, we have
$$
\E(f_n(\tilde{X}^{\{n\}})^{2i})\leq C_{\sup}.
$$
\end{lm}

\begin{proof}
We have
\begin{eqnarray*}
\E(f_n(\tilde{X}^{\{n\}})^{2i})& \leq & C_{\sup} \left( \E(\|\tilde{X}^{\{n\}}\|^{2 i  k})+\E(\|\tilde{X}^{\{n\}}\|^{2 i}) \right)\\
& \leq & C_{\sup} \left( \E(\|\tilde{X}^{\{n\}}\|_{2ik}^{2ik})+\E(\|\tilde{X}^{\{n\}}\|_{2i}^{2i}) \right).
\end{eqnarray*}
Now, by Rosenthal inequality \cite{rosenthal1970subspaces}, we have 
\begin{eqnarray*}
\E(|\tilde{X}_j|^{2ik}) &=& \frac{1}{n^{ik}}\E\left[\left(\sum_{l=1}^n U_j^{(l)}-\mu_j\right)^{2ik}\right] \\
&\leq & \frac{C_{\sup}}{n^{ik}}\max\left(  n \E([U_j^{(1)}-\mu_j]^{2ik}), \left[ n \E([U_j^{(1)}-\mu_j]^2) \right]^{ik}\right) \\
&\leq& C_{\sup}.
\end{eqnarray*}

\end{proof}

\begin{lm}\label{lm_approx_gaussian_indicatrice}
For all $v \subset [1:p]$, $v\neq \emptyset$ and for $i=1,2$, we have
$$
\sup_{n} \E\left(f_n(\tilde{X}^{\{n\}})^i\mathds{1}_{\tilde{X}^{\{n\}}_v \notin [-K,K]^{|v|} }\right)\underset{K \to +\infty}{\longrightarrow}0.
$$
\end{lm}

\begin{proof}
We have
\begin{eqnarray*}
& & \E\left(f_n(\tilde{X}^{\{n\}})^i\mathds{1}_{\tilde{X}^{\{n\}}_v \notin [-K,K]^{|v|} }\right)\\
& \leq & \sqrt{ \E\left(f_n(\tilde{X}^{\{n\}})^{2i} \right)}  \sqrt{\PP( \tilde{X}^{\{n\}}_v \notin [-K,K]^{|v|} )}.
\end{eqnarray*}
By Lemma \ref{lm_approx_gaussian_unif_bound}, $\sup_n \sqrt{ \E\left(f_n(\tilde{X}^{\{n\}})^{2i} \right)} $ is bounded.

Now, since $(\tilde{X}^{\{n\}}_v)_n$ converges in distribution, it is a tight sequence, hence
$$
 \sup_n\PP\left( \tilde{X}_v^{\{n\}}\notin [-K,K]^{|v|}\right)\leq \sup_n \PP(\|\tilde{X}_v^{\{n\}}\|\geq K) \underset{K \to +\infty}{\longrightarrow} 0.
$$
\end{proof}

\begin{lm}\label{lm_approx_gaussian_pointwise}
The sequence $(f_n)_n$ converges pointwise to $Df(\mu)$.
\end{lm}

\begin{proof}
For all $x \in \R$,
$$
 f \left(\frac{x}{\sqrt{n}} +\mu \right)- f (\mu) =Df\left( \mu\right) \frac{x}{\sqrt{n}}+ O\left( \left\| \frac{x}{\sqrt{n}}\right\|^2 \right),
$$
so,
$$
f_n(x)=Df\left( \mu \right) x+ O\left( \frac{\|x\|^2}{\sqrt{n}} \right).
$$
\end{proof}

\subsection{Part 2}\label{section_terme1}
We want to prove that, for all $u \subset[1:p]$, $u\neq \emptyset$, we have
$$
\sup _n |E_{u,n,K}(\tilde{X}^{\{n\}})-E_{u,n}(\tilde{X}^{\{n\}})| \underset{K \to +\infty}{\longrightarrow}0.
$$

We will prove this result for $\emptyset \varsubsetneq u \varsubsetneq [1:p]$, since it is easier for $u= [1:p]$ (see Remark \ref{rmk_u=[1:p]}).\\

We have
\begin{eqnarray*}
& & \Bigg| \int_{\R^{|u|}}\left( \int_{\R^{|-u|}}  f_n(x)  d\PP_{\tilde{X}^{\{n\}}_{-u}|\tilde{X}^{\{n\}}_u=x_u}(x_{-u}) \right)^2 d\PP_{\tilde{X}^{\{n\}}_u}(x_u) \\
& & -  \int_{[-K,K]^{|u|}}\left( \int_{[-K,K]^{|-u|}}  f_n(x) d\PP_{\tilde{X}^{\{n\}}_{-u}|\tilde{X}^{\{n\}}_u=x_u}(x_{-u}) \right)^2 d\PP_{\tilde{X}^{\{n\}}_u}(x_u) \Bigg| \\
&\leq & \int_{([-K,K]^{|u|})^c} \left( \int_{\R^{|-u|}}  f_n(x)  d\PP_{\tilde{X}^{\{n\}}_{-u}|\tilde{X}^{\{n\}}_u=x_u}(x_{-u}) \right)^2 d\PP_{\tilde{X}^{\{n\}}_u}(x_u) \\
& &+ \int_{[-K,K]^{|u|}}\Bigg|\left( \int_{\R^{|-u|}}  f_n(x) d\PP_{\tilde{X}^{\{n\}}_{-u}|\tilde{X}^{\{n\}}_u=x_u}(x_{-u}) \right)^2\\
 & & -\left( \int_{[-K,K]^{|-u|}}  f_n(x) d\PP_{\tilde{X}^{\{n\}}_{-u}|\tilde{X}^{\{n\}}_u=x_u}(x_{-u}) \right)^2  \Bigg|  d\PP_{\tilde{X}^{\{n\}}_u}(x_u). 
\end{eqnarray*}
We have to bound the two summands of the previous upper-bound.

The first term converges to $0$ by Lemma \ref{lm_approx_gaussian_indicatrice}. Let us bound the second term.
By mean-value inequality with the square function, we have
\begin{eqnarray*}
& &\int_{[-K,K]^{|u|}} \Bigg|\left( \int_{\R^{|-u|}}  f_n(x) d\PP_{\tilde{X}^{\{n\}}_{-u}|\tilde{X}^{\{n\}}_u=x_u}(x_{-u}) \right)^2\\
& & -\left( \int_{[-K,K]^{|-u|}}  f_n(x) d\PP_{\tilde{X}^{\{n\}}_{-u}|\tilde{X}^{\{n\}}_u=x_u}(x_{-u}) \right)^2  \Bigg| d\PP_{\tilde{X}^{\{n\}}_u}(x_u) \\
& \leq & 2 \int_{[-K,K]^{|u|}} \left( \int_{\R^{|-u|}}  |f_n(x)| d\PP_{\tilde{X}^{\{n\}}_{-u}|\tilde{X}^{\{n\}}_u=x_u}(x_{-u}) \right) \\
& & \left| \int_{\R^{|-u|}} \mathds{1}_{x_{-u}\notin [-K,K]^{|-u|}}  f_n(x) d\PP_{\tilde{X}^{\{n\}}_{-u}|\tilde{X}^{\{n\}}_u=x_u}(x_{-u})  \right| d\PP_{\tilde{X}^{\{n\}}_u}(x_u)\\
& \leq &2 \int_{[-K,K]^{|u|}} \left( \int_{\R^{|-u|}}  |f_n(x)| d\PP_{\tilde{X}^{\{n\}}_{-u}|\tilde{X}^{\{n\}}_u=x_u}(x_{-u}) \right)\\
& & \times \left(\int_{\R^{|-u|}} \mathds{1}_{x_{-u}\notin [-K,K]^{|-u|}}  |f_n(x)| d\PP_{\tilde{X}^{\{n\}}_{-u}|\tilde{X}^{\{n\}}_u=x_u}(x_{-u}) \right)  d\PP_{\tilde{X}^{\{n\}}_u}(x_u)\\
& \leq & 2\sqrt{\E(\E(|f_n(\tilde{X}^{\{n\}})| \;| \tilde{X}^{\{n\}}_u)^2) }\\ & &  
\times \sqrt{\int_{\R^{|u|}}\left(\int_{\R^{|-u|}} \mathds{1}_{x_{-u}\notin [-K,K]^{|-u|}} |f_n(x)| d\PP_{\tilde{X}^{\{n\}}_{-u}|\tilde{X}^{\{n\}}_u=x_u}(x_{-u})  \right)^2 d\PP_{\tilde{X}^{\{n\}}_u}(x_u)}.
\end{eqnarray*}
Now, $\E(\E(|f_n(\tilde{X}^{\{n\}})| \;| \tilde{X}^{\{n\}}_u)^2) \leq \E(f_n(\tilde{X}^{\{n\}})^2) $ which is bounded by Lemma \ref{lm_approx_gaussian_unif_bound} and the other term converges to $0$ uniformly on $n$ by Lemma \ref{lm_approx_gaussian_indicatrice}. 

\begin{rmk}\label{rmk_u=[1:p]}
In the case where $u=[1:p]$, it is much simpler, since
\begin{eqnarray*}
\E(f_n(\tilde{X}^{\{n\}})^2)-\E(f_n(\tilde{X}^{\{n\}})^2\mathds{1}_{\tilde{X}^{\{n\}} \in [-K,K]^p})= \E(f_n(\tilde{X}^{\{n\}})^2\mathds{1}_{\tilde{X}^{\{n\}} \notin [-K,K]^p}),
\end{eqnarray*}
which converges to 0 uniformly on $n$ when $K\to +\infty$ by Lemma \ref{lm_approx_gaussian_indicatrice}.
\end{rmk}

\subsection{Part 3}\label{section_terme2}
Let $K \in \R_+^*$ and $u\subset [1:p]$ such that $u\neq \emptyset$.
We want to prove that 
$$
|E_{u,n,K}(X^*)-E_{u,n,K}(\tilde{X}^{\{n\}})|\underset{n \to +\infty}{\longrightarrow} 0.
$$

The case $u=[1:p]$ is much easier (see Remark \ref{rmk_u=[1:p]2}), hence, assume that $\emptyset \varsubsetneq u \varsubsetneq [1:p]$.
Since $K$ is fixed, the probability density function $f_{X^*}$ of $X^*$ is lower-bounded by $a>0$ on $[-K,K]^p$. Let $\varepsilon_n:=\max_{\emptyset \varsubsetneq u\subset[1:p]}\sup_{x\in \R^p}|f_{X^*}(x)-f_{\tilde{X}^{\{n\}}}(x)|$. Using local limit theorem (see Theorem 19.1 of \cite{bhattacharya1986normal} or \cite{shervashidze1971uniform}), $\varepsilon_n \underset{n \to +\infty}{\longrightarrow}0$. We assume that $n$ is large enough such that $\varepsilon_n\leq \frac{a}{2}$. Let $b< +\infty$ be the maximum of $f_{X^*}$.

We have
\begin{eqnarray*}
& & |E_{u,n,K}(X^*)-E_{u,n,K}(\tilde{X}^{\{n\}})|\\
& \leq & \int_{[-K,K]^{|u|}} \Bigg| \left( \int_{[-K,K]^{|-u|}} f_n(x) \frac{f_{X^*}(x)}{f_{X_u^*}(x_u)}dx_{-u} \right)^2 \\
& & - \left( \int_{[-K,K]^{|-u|}} f_n(x) \frac{f_{\tilde{X}^{\{n\}}}(x)}{f_{\tilde{X}_u^{\{n\}}}(x_u)}dx_{-u} \right)^2 \Bigg| f_{X_u^*}(x_u) dx_u\\
&  &+  \int_{[-K,K]^{|u|}}  \left( \int_{[-K,K]^{|-u|}} f_n(x) \frac{f_{\tilde{X}^{\{n\}}}(x)}{f_{\tilde{X}_u^{\{n\}}}(x_u)}dx_{-u} \right)^2 |f_{X_u^*}(x_u)-f_{\tilde{X}_u^{\{n\}}}(x_u)| dx_u .
\end{eqnarray*}
Hence, we have to prove the convergence of the two summands in the previous upper-bound. For the second term, it suffices to remark that
$$
|f_{X_u^*}(x_u)-f_{\tilde{X}_u^{\{n\}}}(x_u)|\leq \varepsilon_n\leq \frac{2\varepsilon_n}{a}f_{\tilde{X}_u^{\{n\}}}(x_u).
$$
Hence, the second term is smaller than $\frac{2\varepsilon_n}{a}\E(f_n(\tilde{X}^{\{n\}})^2)$ that converges to 0. It remains to prove that the first term converges to $0$. By mean-value inequality, we have 
\begin{eqnarray*}
& & \int_{[-K,K]^{|u|}} \Bigg| \left( \int_{[-K,K]^{|-u|}} f_n(x) \frac{f_{X^*}(x)}{f_{X_u^*}(x_u)}dx_{-u} \right)^2 \\
& & - \left( \int_{[-K,K]^{|-u|}} f_n(x) \frac{f_{\tilde{X}^{\{n\}}}(x)}{f_{\tilde{X}^{\{n\}}_u}(x_u)}dx_{-u} \right)^2 \Bigg| f_{X_u^*}(x_u) dx_u\\
& \leq & 2\int_{[-K,K]^{|u|}}\left( \int_{[-K,K]^{|-u|}} |f_n(x)| \max\left(\frac{f_{X^*}(x)}{f_{X_u^*}(x_u)}, \frac{f_{\tilde{X}^{\{n\}}}(x)}{f_{\tilde{X}^{\{n\}}_u}(x_u)}\right)dx_{-u} \right)\\
& & \times \left( \int_{[-K,K]^{|-u|}} |f_n(x)|  \left|\frac{f_{X^*}(x)}{f_{X_u^*}(x_u)}- \frac{f_{\tilde{X}^{\{n\}}}(x)}{f_{\tilde{X}^{\{n\}}_u}(x_u)}\right|dx_{-u} \right)f_{X_u^*}(x_u) dx_u.
\end{eqnarray*}
Now,
\begin{eqnarray*}
\left|\frac{f_{X^*}(x)}{f_{X_u^*}(x_u)}- \frac{f_{\tilde{X}^{\{n\}}}(x)}{f_{\tilde{X}^{\{n\}}_u}(x_u)}\right| & \leq & \frac{|f_{X^*}(x)-f_{\tilde{X}^{\{n\}}}(x) |}{f_{X_u^*}(x_u)}+ f_{\tilde{X}^{\{n\}}}(x) \left|\frac{1}{f_{X_u^*}(x_u)}- \frac{1}{ f_{\tilde{X}^{\{n\}}_u}(x_u)}\right|\\
& \leq & \frac{|f_{X^*}(x)-f_{\tilde{X}^{\{n\}}}(x) |}{f_{X_u^*}(x_u)}+ f_{\tilde{X}^{\{n\}}}(x) \frac{4}{a^2} \left| f_{X_u^*}(x_u)-  f_{\tilde{X}^{\{n\}}_u}(x_u)\right|\\
& \leq & \frac{\varepsilon_n}{f_{X_u^*}(x_u)}+f_{\tilde{X}^{\{n\}}}(x) \frac{4}{a^2} \varepsilon_n\\
& \leq & \frac{\varepsilon_n}{f_{X_u^*}(x_u)}+f_{X^*}(x) \frac{8}{a^2} \varepsilon_n\\
& \leq & \frac{\varepsilon_n}{a}\frac{f_{X^*}(x)}{f_{X_u^*}(x_u)}+ \frac{8b}{a^2}\varepsilon_n \frac{f_{X^*}(x)}{f_{X_u^*}(x_u)}\\
&\leq & C_{\sup} \varepsilon_n  \frac{f_{X^*}(x)}{f_{X_u^*}(x_u)}.
\end{eqnarray*}
Hence, for $n$ large enough such that $ C_{\sup} \varepsilon_n \leq 1$, we have
\begin{eqnarray*}
& & \int_{[-K,K]^{|u|}} \Bigg| \left( \int_{[-K,K]^{|-u|}} f_n(x) \frac{f_{X^*}(x)}{f_{X_u^*}(x_u)}dx_{-u} \right)^2 \\
& & - \left( \int_{[-K,K]^{|-u|}} f_n(x) \frac{f_{\tilde{X}^{\{n\}}}(x)}{f_{\tilde{X}^{\{n\}}_u}(x_u)}dx_{-u} \right)^2 \Bigg| f_{X_u^*}(x_u) dx_u\\
& \leq & 2\int_{[-K,K]^{|u|}}\left( \int_{[-K,K]^{|-u|}} |f_n(x)| 2 \frac{f_{X^*}(x)}{f_{X_u^*}(x_u)}dx_{-u} \right)\\
& & \times \left( \int_{[-K,K]^{|-u|}} |f_n(x)|   C_{\sup} \varepsilon_n  \frac{f_{X^*}(x)}{f_{X_u^*}(x_u)}  dx_{-u} \right)f_{X_u^*}(x_u) dx_u\\ \\
& \leq & C_{\sup} \varepsilon_n \E(f_n(X^*)^2),
\end{eqnarray*}
that converges to $0$.

\begin{rmk}\label{rmk_u=[1:p]2}
If $u=[1:p]$, it suffices to remark that 
$$
|f_{X^*}(x)-f_{\tilde{X}^{\{n\}}}(x)|\leq \varepsilon_n\leq \frac{\varepsilon_n}{a}f_{X^*}(x).
$$
Thus,
\begin{eqnarray*}
& & |E_{u,n,K}(X^*)-E_{u,n,K}(\tilde{X}^{\{n\}})|\\
& \leq & \int_{[-K,K]^p} f_n(x)^2 |f_{X^*}(x)-f_{\tilde{X}^{\{n\}}}(x)| dx\\
& \leq & \frac{\varepsilon_n}{a} \E(f_n(X^*)^2)\\
& \leq & C_{\sup} \varepsilon_n.
\end{eqnarray*}
\end{rmk}

\subsection{Part 4}\label{section_term3}
Let us prove that
$$
\E(f_n(\tilde{X}^{\{n\}}))-\E(f_n(X^*))\longrightarrow 0
$$
By lemma \ref{lm_approx_gaussian_indicatrice}, we have
$$
\sup_n\left| \E(f_n(\tilde{X}^{\{n\}})- \E(f_n(\tilde{X}^{\{n\}})\mathds{1}_{\tilde{X}^{\{n\}} \in [-K,K]^p}) \right| \underset{K \to \infty}{\longrightarrow}0
$$
Let $\varepsilon>0$ and let $K$ such that
$$
\sup_n\left| \E(f_n(\tilde{X}^{\{n\}})- \E(f_n(\tilde{X}^{\{n\}})\mathds{1}_{\tilde{X}^{\{n\}} \in [-K,K]^p}) \right|< \frac{\varepsilon}{3}
$$
and
$$
\sup_n\left| \E(f_n({X^*})- \E(f_n({X^*})\mathds{1}_{X^* \in [-K,K]^p}) \right| < \frac{\varepsilon}{3}.
$$
By local limit theorem, we have
$$
\left| \E(f_n(\tilde{X}^{\{n\}})\mathds{1}_{\tilde{X}^{\{n\}} \in [-K,K]^p}) - \E(f_n({X^*})\mathds{1}_{X^* \in [-K,K]^p})\right| \underset{n \to +\infty}{\longrightarrow}0.
$$

Thus, for all $u \subset [1:p]$, we have
$$
\V(\E(f_n(\tilde{X}^{\{n\}})|\tilde{X}^{\{n\}}_u))- \V(\E(f_n(X^*)|X^*_u))\underset{n \to +\infty}{\longrightarrow}0.
$$

\subsection{Conclusion}\label{section_approx_gaussien_conclu}

To prove the convergence of the Shapley effects, it suffices to prove the $\V(f_n(X^*))$ is lower-bounded. Hence, we show that $\V(f_n(X^*))$ converges to $\V(Df(\mu) X^*)$. Let $i=1,2$ and let $\varepsilon>0$. By Lemma \ref{lm_approx_gaussian_indicatrice}, let $K$ such that 
$$
\sup_n\E(f_n(X^*)^i\mathds{1}_{X^* \notin [-K,K]^p})\leq \frac{\varepsilon}{3},\;\;\E([Df(\mu)X^*]^i\mathds{1}_{X^* \notin [-K,K]^p})\leq \frac{\varepsilon}{3}.
$$
By Lemmas \ref{lm_gaussian_approx1} and \ref{lm_approx_gaussian_pointwise} and by dominated convergence theorem, we have :
$$
\E(f_n(X^*)^i\mathds{1}_{X^* \in [-K,K]^p})\underset{n\to +\infty}{\longrightarrow}\E([Df(\mu)X^*]^i\mathds{1}_{X^* \in [-K,K]^p}).
$$
Hence, $\V(f_n(X^*))$ converges to $\V(Df(\mu) X^*)$. Thus, for all $u\subset [1:p]$
$$
S_u^{cl}(\tilde{X}^{\{n\}},f_n)- S_u(X^*,f_n) \underset{n \to +\infty}{\longrightarrow}0.
$$
Hence,
$$
\left\| \eta(\tilde{X}^{\{n\}},f_n)- \eta(X,f_n)  \right\|\underset{n \to +\infty}{\longrightarrow}0.
$$
\subsection{Proof of Corollary \ref{corol_approx_gaussien}}

Since $\widehat{X}^{\{n'\}\prime}\overset{a.s}{\underset{n\to +\infty}{\longrightarrow}}\mu$ and $\widehat{\Sigma}^{\{n''\}\prime}\overset{a.s}{\underset{n\to +\infty}{\longrightarrow}}\Sigma$, it suffices to prove that, if $(x^{\{n\}})_n$ converges to $\mu$, and $(\Sigma^{\{n\}})_n$ converges to $\Sigma$, we have
$$
\left\|\eta(\widehat{X}^{\{n\}},f)- \eta(X^{*n},\tilde{f}_{1,h^{\{n\}},x^{\{n\}}}^{\{n\}})\right\| \underset{n \to +\infty}{\longrightarrow}0,
$$
where $X^{*n}$ is a random vector with distribution $\mathcal{N}(\mu,\Sigma^{\{n\}})$.
Let $(x^{\{n\}})_n$ and $(\Sigma^{\{n\}})_n$ be such sequences. Recall that
$$
\left\| \eta(\tilde{X}^{\{n\}},f_n)- \eta(X^*,f_n)  \right\|\underset{n \to +\infty}{\longrightarrow}0,
$$
where $X^*\sim \mathcal{N}(0,\Sigma)$, that is
$$
\left\| \eta\left(  \widehat{X}^{\{n\}}, f\right) -\eta\left(  X^{\{n\}}, f\right)   \right\| \underset{n \to +\infty}{\longrightarrow}0,
$$
where $X^{\{n\}}\sim \mathcal{N}(\mu, \frac{1}{n}\Sigma)$.
Hence, we have to prove that
$$
\left\|\eta(X^{\{n\}},f)- \eta(X^{*n},\tilde{f}_{1,h^{\{n\}},x^{\{n\}}}^{\{n\}})\right\| \underset{n \to +\infty}{\longrightarrow}0.
$$
By Propositions \ref{prop_linear1} and Proposition \ref{prop_linear2}, remark that $\eta(X^{\{n\}},f)$ converges to $\eta(X^*,f_1)$. Moreover, 
$$
\eta(X^{*n},\tilde{f}_{1,h^{\{n\}},x^{\{n\}}}^{\{n\}})= \eta(X^{*n}+x^{\{n\}}-\mu^{\{n\}},\tilde{f}_{1,h^{\{n\}},x^{\{n\}}}^{\{n\}})\underset{n\to +\infty}{\longrightarrow}\eta(X^*,f_1),
$$
by Corollary \ref{corol_linear_Dh1}, that concludes the proof.
\end{document}